\documentclass{daj}
 
\usepackage{amsmath, amsfonts, amsthm, amssymb, bm, graphicx, enumerate}

\newtheorem{theorem}{Theorem}[section]       
\newtheorem{lemma}[theorem]{Lemma}
\newtheorem{corollary}[theorem]{Corollary}

\newtheorem{proposition}[theorem]{Proposition} 
 
\newtheorem*{main theorem}{Main Theorem}

\theoremstyle{remark}     
\newtheorem*{rem}{Remark}   
\newtheorem*{remn}{Remark on notation} 
\theoremstyle{definition}  
\newtheorem{definition}[theorem]{Definition}

\usepackage{color}

\def\l{\ell}
\def\sing{{\mathrm{s}}}

\def\cQ{\mathcal{Q}}

\def\F{\mathbb{F}}  
\def\i{i}
\def\T{\mathbb{T}}         
\def\N{\mathbb{N}}     
\def\Q{\mathbb{Q}}          
\def\R{\mathbb{R}}     
\def\Z{\mathbb{Z}}
     
\def\C{\mathbb{C}}  
\def\P{\mathcal{P}}

\def\gradient{\nabla} 
\def\grad{\gradient}   
\def\bar#1{\overline{#1}} 

\def\cont{\textnormal{cont}} 
\def\implies{\Longrightarrow}  

\def\calL{{\mathcal L}}

\def\calX{{\mathcal X}}

\def\bsi{\boldsymbol{i}}
\def\bsr{\boldsymbol{r}}
\def\bsa{{\boldsymbol{a}}}

\def\bsx{\boldsymbol{x}}
\def\bsz{\boldsymbol{z}}
\def\bsn{\boldsymbol{n}}
\def\bsm{\boldsymbol{m}}
\def\bss{\boldsymbol{s}}

\def\bbA{\mathbb{A}} 
\def\bbP{\mathbb{P}}
\def\calG{\mathcal{G}} 
\def\QQbar{\bar{\Q}}
\def\ZZbar{\bar{\Z}}

\def\Lhat{{\widehat{L}}}

\def\Vhat{{\widehat{V}}}
\def\What{{\widehat{W}}}

\def\bszero{\boldsymbol{0}}

\def\calLhat{\widehat{\calL}}

\def\Gal{{\operatorname{Gal}}}
\def\ns{{\operatorname{ns}}}

\def\Gal{\operatorname{Gal}}
\def\Res{\operatorname{Res}}

\def\ns{{\operatorname{ns}}}

\dajAUTHORdetails{%
  title = {Multivariate Polynomial Values\\ in Difference Sets}, 
  author = {John R. Doyle and Alex Rice},
  plaintextauthor = {John R. Doyle, Alex Rice},
  plaintexttitle = {Multivariate Polynomial Values in Difference Sets}, 
  runningtitle = {Multivariate Polynomial Values in Difference Sets}, 
}   

\dajEDITORdetails{%
   year={2021},
   number={11},
   received={14 September 2020},   
   published={8 September 2021},  
   doi={10.19086/da.27995},       
}   

\begin{document} 

\begin{frontmatter}[classification=text]

 
\author[john]{John R. Doyle\thanks{Partially supported by NSF grant DMS-2001486.}}
\author[alex]{Alex Rice}

\begin{abstract}
For $\l\geq 2$ and $h\in \Z[x_1,\dots,x_{\l}]$ of degree $k\geq 2$, we show that every subset $A\subseteq \{1,2,\dots,N\}$ lacking nonzero differences in $h(\Z^{\l})$ satisfies $|A|\ll_h Ne^{-c(\log N)^{\mu}}$, where $c=c(h)>0$, $\mu=[(k-1)^2+1]^{-1}$ if $\l=2$, and $\mu=1/2$ if $\l\geq 3$, provided $h(\Z^{\l})$ contains a multiple of every natural number and $h$ satisfies certain nonsingularity conditions. We also explore these conditions in detail, drawing on a variety of tools from algebraic geometry.
\end{abstract}
\end{frontmatter}
 
\section{Introduction}    

Originating with conjectures of Erd\H{o}s and Lov\'asz, an extensive literature has developed over the past several decades concerning the existence of particular differences within dense sets of integers. For sets $A,B\subseteq \Z$, we define the sum and difference sets, respectively, as usual by $A\pm B= \{a \pm b: a\in A, b\in B\}$, and we also define the following threshold.
\begin{definition}For $X\subseteq \Z$ and $N\in \N$, we define $D(X,N)=\max \left\{|A|: A\subseteq [1,N], \ (A-A)\cap X \subseteq \{0\} \right\}. $  
\end{definition}

\noindent We use $[1,N]$ to denote $\{1,2,\dots,N\}$ and $|A|$ to denote the size of a finite set $A$. To clarify, $D(X,N)$ is the threshold such that any subset of $\{1,2,\dots,N\}$ with more than $D(X,N)$ elements necessarily contains two distinct elements that differ by an element of $X$. As an introductory offering prior to extensive discussions of history, motivation, notation, and terminology, a very special case of our results in this paper is the following:  

\begin{theorem} \label{introthm} Suppose $h\in \Z[x,y]$ is a homogeneous polynomial of degree $k\geq 2$. If $\Delta(h) \neq 0$, then \begin{equation} \label{2vb} D(h(\Z^{2}),N) \ll_h N e^{-c(\log N)^{\mu}}, \end{equation} where $c=c(h)>0$ and $\mu=[(k-1)^2+1]^{-1}$.
\end{theorem} 
\noindent Here $\Delta$ denotes the usual homogeneous discriminant, and we use $\ll$ to denote ``less than a constant times'', with subscripts indicating on what parameters, if any, the implied constant depends. We take the same convention with subscripts on Big O notation. Theorem \ref{introthm} follows from Corollary \ref{2vcor} and our main result, Theorem \ref{more}, of which we discuss various improvements and important special cases throughout Section \ref{maindr}.
\subsection{Background} Lov\'asz asked whether a set of positive upper density must contain two distinct elements that differ by a perfect square, or equivalently whether $D(S,N)=o(N)$, where $S=\{n^2: n\in \N\}$. Similarly, Erd\H{o}s conjectured that $D(\mathcal{P}-1,N)=o(N)$, where $\mathcal{P}-1=\{p-1: p \text{ prime}\}$. Furstenberg \cite{Furst} verified the former using ergodic methods, specifically his correspondence principle, in the same paper in which he provided the second known proof of Szemer\'edi's Theorem on arithmetic progressions. Independently and concurrently, S\'ark\"ozy (\cite{Sark1}, \cite{Sark3}) verified both conjectures with a Fourier analytic density increment argument driven by the Hardy-Littlewood circle method. Further, S\'ark\"ozy's results included quantitative information, showing $ D(S,N) \ll_{\epsilon} N(\log N)^{-1/3+\epsilon}$ and $ D(\mathcal{P}-1,N)\ll_{\epsilon} N (\log\log N)^{-2+\epsilon}$ for every $\epsilon>0.$ 

These results have been incrementally improved and generalized in multiple ways, both through tightening of the quantitative bounds and expansion of the possibilities for the set $X$ of prohibited differences. Regarding the former, Pintz, Steiger and Szemer\'edi \cite{PSS} utilized a more elaborate Fourier analytic strategy to show \begin{equation}\label{PSSb} D(S,N)\ll N (\log N)^{-c\log\log\log\log N} \end{equation} for a constant $c>0$. 
 
Dramatically improving S\'ark\"ozy's original bound, Ruzsa and Sanders \cite{Ruz} showed \begin{equation} \label{rsb} D(\mathcal{P}-1,N)\ll Ne^{-c(\log N)^{\mu}}\end{equation} with $\mu=1/4$, recently improved to $\mu=1/3$ by Wang \cite{wang}. Regarding alternative choices for the set of prohibited differences, one must first consider obvious local obstructions. For example, we consider $\mathcal{P}-1$, rather than $\mathcal{P}$, because  $\P\cap 4\Z = \emptyset$ implies $D(\mathcal{P},N)\geq \lceil N/4 \rceil$ by taking $A$ to be a congruence class modulo $4$. Analogously, if $h\in \Z[x]$ and $h(\Z)$ contains no multiples of $q\in \N$, then $D(h(\Z),N)\geq \lceil N/q\rceil$. Therefore, for even a qualitative $o(N)$ result, it is clearly necessary that $h(\Z)$ contains a nonzero multiple of every $q\in \N$, in which case we say that $h$ is an \textit{intersective polynomial}. Examples of intersective polynomials include any nonzero polynomial with an integer root or a collection of rational roots with coprime denominators. However, there are also intersective polynomials with no rational roots, such as $(x^3-19)(x^2+x+1)$.  

Balog, Pelik\'an, Pintz, and Szemer\'edi \cite{BPPS} extended (\ref{PSSb}) with $S$ replaced by $\{n^k: n\in \N\}$ for any fixed $k\in \N$. For a general univariate intersective polynomial, Kamae and Mendes-France \cite{KMF} established the qualitative $o(N)$ result, the first quantitative bounds were due to Lucier \cite{Lucier}, and the second author \cite{ricemax} fully extended (\ref{PSSb}). In a recent preprint, Bloom and Maynard \cite{BloomMaynard} both simplified and improved the ideas of \cite{PSS}, using a more traditional density increment to establish \begin{equation}\label{BMnew} D(S,N)\ll N (\log N)^{-c\log\log\log N} \end{equation} for a constant $c>0$, which is currently the best-known bound for the original square difference question. Further, the methods of \cite{BloomMaynard} are completely compatible with those of \cite{ricemax}, so in fact (\ref{BMnew}) should hold for the full class of intersective polynomials.  For other intermediate and related results, as well as alternative proofs, the reader may refer to (in chronological order)  \cite{Green}, \cite{Slip},  \cite{Lucier2},  \cite{LM}, \cite{lipan},  \cite{Lyall}, \cite{HLR}, \cite{Rice}, and \cite{taoblog}. 

Also in \cite{ricemax}, the second author showed that if $g,h\in \Z[x]$ are intersective polynomials, then \begin{equation} \label{splitb} D(g(\Z)+h(\Z), N) \ll_{g,h} N e^{-c(\log N)^{\mu}}, \end{equation} where $c=c(g,h)>0$ and $\mu=\mu(\deg(g),\deg(h))>0$. Further, the second author \cite{Ricebin} considered the simplest nontrivial case of a non-diagonal multivariate polynomial, showing that for a binary quadratic form $h(x,y)=ax^2+bxy+cy^2\in \Z[x,y]$ with $b^2-4ac\neq 0$, we have \begin{equation} \label{binb} D(h(\Z^2), N) \ll_{h} N e^{-c\sqrt{\log N}}. \end{equation} 
 
\subsection{Motivation} As outlined in Section 2.4 of \cite{ricemax}, the quoted upper bounds in the previous section, all of which result from adaptations of the two aforementioned Fourier analytic arguments developed in \cite{Sark1} and \cite{PSS}, respectively, are partially determined by the degree of decay in local exponential averages similar to \begin{equation}\label{gsintro} q^{-1}\sum_{s=0}^{q-1}e^{2\pi i h(s)a/q}. \end{equation}
The best general upper bound for (\ref{gsintro}) is of the order $q^{-1/k}$ where $k=\deg(h)$, but the elaborate double iteration method developed in \cite{PSS}, and the simplified improvement developed in \cite{BloomMaynard}, which lead to upper bounds like (\ref{PSSb}) and (\ref{BMnew}), require decay at or near $q^{-1/2}$, which we refer to as \textit{square-root cancellation}. Inspired by \cite{BPPS}, the second author \cite{ricemax} eliminated this discrepancy for $k> 2$ in the general case by employing a polynomial-specific sieve to the set of considered inputs that, roughly speaking, reduced the issue to estimating (\ref{gsintro}) at prime moduli, for which the desired square-root cancellation is a well-known result of Weil. This sieve technique can be thought of as a bridge from the integer setting to the best available exponential sum estimates over finite fields.
 
Ruzsa and Sanders \cite{Ruz}, and later Wang \cite{wang}, were able to adapt the more traditional density increment method to establish (\ref{rsb}), which is a stronger type of upper bound as compared with (\ref{PSSb}) or (\ref{BMnew}),  based on two key factors: the high degree of decay in the relevant exponential averages, which are modifications of \begin{equation*} \phi(q)^{-1}\sum_{\substack{s=0 \\(s,q)=1}}^{q-1} e^{2\pi i s/q}=\frac{\mu(q)}{\phi(q)}, \end{equation*} and the careful analysis of the distribution of primes in arithmetic progressions, including the consideration of exceptional zeros of Dirichlet $L$-functions. In the polynomial setting, the distribution of inputs in arithmetic progressions is not as delicate of an issue, though it does rear its head when employing a sieve, but this level of local decay is out of reach with a single variable. Specifically, bounds like (\ref{2vb}) from the density increment require decay at or near $q^{-1}$ (more specifically, $q^{-1}$ times a function of average value at most polylogarithmic in $q$, and the exponent $\mu$ depends on the power of the logarithm), which we refer to as \textit{q-cancellation}. 

While the image of a multivariate intersective polynomial does not necessarily contain the image of a univariate intersective polynomial, it is the case that, by only exploiting cancellation in one variable, the methods of \cite{ricemax} and \cite{BloomMaynard} can be adapted to show that (\ref{BMnew}) holds for such an image, so upper bounds in the multivariate setting are only novel if they are stronger than (\ref{BMnew}). The observation made in \cite{ricemax} to establish (\ref{splitb}) was a rather simple one: if we consider differences of the form $g(m)+h(n)$, then the relevant exponential sum factors into a product, our sieve gives square-root cancellation in each variable, and these combine to give $q$-cancellation. However, this observation does not fully generalize to the case of a single polynomial in several variables with nonzero cross-terms. In particular, simple examples like $h(x,y)=(x+y)^2$ make it clear that one cannot always exploit cancellation in each variable, so some sort of nonsingularity assumption is required. 

In the setting of binary quadratic forms, the natural assumption is nonzero discriminant, and since sieving is not required to get square-root cancellation from each variable in degree $2$, the adaptation of the usual density increment is relatively straightforward, as done in \cite{Ricebin} to establish (\ref{binb}). Section 2 of \cite{Ricebin} provides a helpful description of the density increment method in a simpler, sieve-free context. 

For higher degrees, the sieve technique can indeed be adapted to the multivariate setting, which leads us toward the best available estimates on exponential sums for multivariate polynomials over finite fields, due to Deligne \cite{Deligne} in his proof of the Weil conjectures, and their associated nonsingularity assumptions. Recall that $\bbA^n$ and $\bbP^n$ denote $n$-dimensional affine and projective space, respectively.

\begin{definition} Suppose $F$ is a field, $\ell \in \N$, and $g\in F[x_1,\dots,x_{\l}]$ is a homogeneous polynomial. We say that $g$ is \textit{smooth} if the vanishing of $g$ defines a smooth hypersurface in $\mathbb{P}^{\ell-1}$ (as opposed to $\bbA^{\l}$). In other words, $g$ is smooth if the system $ g(\bsx)=\frac{\partial g}{\partial x_1}(\bsx)=\cdots=\frac{\partial g}{\partial x_{\l}}(\bsx)=0 $ has no solution besides $x_1=\cdots=x_{\l}=0$ in $\bar{F}^{\l}$, where the bar indicates the algebraic closure. For a general polynomial $h\in F[x_1,\dots,x_{\ell}]$ with  $h=\sum_{i=0}^k h^i$, where $h^i$ is homogeneous of degree $i$ and $h^k\neq 0$, we say that $h$ is \textit{Deligne} if the characteristic of $F$ does not divide $k$ and $h^k$ is smooth.
\end{definition} 

\begin{remn} For the remainder of the paper, we take the notational convention that, for a polynomial $h$, $h^i$ denotes the degree-$i$ homogeneous part of $h$, as opposed to $h$ raised to the $i$-th power.

\end{remn} 

\begin{theorem}[Theorem 8.4, \cite{Deligne}] \label{delmain} Suppose $\l \in \N$ and $p\in \P$. If $h\in \mathbb{F}_p[x_1,\dots,x_{\l}]$ is Deligne, then \begin{equation*} \left|\sum_{\bsx\in \mathbb{F}_p^{\l}} e^{2\pi i h(\bsx)/p} \right| \leq (\deg(h)-1)^{\l} p^{\l/2}. \end{equation*}

\end{theorem}  

This estimate provides a guide, but additional consideration is required to develop sufficient conditions on a multivariate polynomial for an application of Theorem \ref{delmain} that is compatible enough with the density increment procedure to establish an upper bound like (\ref{2vb}). We explore these details in Section \ref{maindr}. 

\subsection{Lower bounds and a special case}\label{lbsec} In all the nontrivial cases we have explored, there is a large gap in the best-known upper and lower bounds for $D(X,N)$. For an intersective polynomial $h\in \Z[x]$, all known lower bounds with $X=h(\Z)$ are of order $N^c$ for some $c<1$. The greedy algorithm gives $c=1-1/\deg(h)$, and higher values of $c$ are known for monomials (see \cite{Ruz2} and \cite{Lewko}) and certain other polynomials divisible by $x^2$ (due to Younis \cite{younis}, and explored from an algebraic number theory perspective by Wessel \cite{wessel}). For $X=\mathcal{P}-1$, the gap is even larger, and the best-known lower bound is of the form $N^{o(1)}$ (see \cite{Ruz3}). Younis \cite{younis} established lower bounds for certain homogeneous multivariate polynomials, including $D(S+S,N)\gg \sqrt{N}$, where $S$ is the set of squares.  All of these results are descended from methods of Ruzsa that transfer examples from the modular setting to the integer setting. In the absence of stronger lower bounds, the greedy algorithm gives $D(X,N)\geq (N-1)/(|X\cap[-N,N]|+1)$ for any set $X\subseteq \Z$ (see \cite{Lyall}). 

As an aside, one very special case where stronger upper bounds on $D(X,N)$ are available, and where the upper and lower bounds can be relatively close, is the case when $X$ is itself, or at least contains, a difference set. Specifically, if $Y\subseteq \{1,\dots, N\}$ and $X=Y-Y$, then for a set $A\subseteq\{1,\dots,N\}$ satisfying $(A-A)\cap X \subseteq \{0\}$, we have $a+y\neq a'+y'$ for all $a,a'\in A$ and $y,y'\in Y$ with $(a,y)\neq (a',y')$. In particular, the map $(a,y)\mapsto a+y$ into $\{1,\dots,2N\}$ is an injection, so $|A||Y|\leq 2N$, and hence $D(X,N)\leq 2N/|Y|$, while the greedy algorithm gives $D(X,N)\gg N/|X| \geq N/|Y|^2$. For an example relating to our discussion of multivariate polynomials, if $X$ is the set of differences of $k$-th powers for a fixed $k\in \N$, then we have $D(X,N)\ll N^{1-1/k}$, but this observation does not immediately generalize beyond the case where $X\supseteq Y-Y$. 

\section{Main definitions and results} \label{maindr}

The density increment procedure takes as input a set $A\subseteq \{1,2,\dots,N\}$ lacking nonzero differences in the image of a polynomial $h$, and produces a new, denser subset of a slightly smaller interval lacking nonzero differences in the image of a potentially modified polynomial. The following definition keeps track of the changes in the polynomial over the course of the iteration.

\begin{definition} Fix $\l\in \N$. As in the univariate setting, we say that $h\in \Z[x_1,\dots,x_{\l}]$ is \textit{intersective} if $h(\Z^{\l})$ contains a nonzero multiple of every $q\in \N$. Equivalently, $h$ is intersective if it is not identically zero and has a root in $\Z_p^{\l}$ for every prime $p$, where $\Z_p$ denotes the $p$-adic integers.
\end{definition}

\noindent Suppose $h\in \Z[x_1,\dots,x_{\ell}]$ is an intersective polynomial and fix, for each prime $p$, $\bsz_p\in \Z_p^{\ell}$ with $h(\bsz_p)=0$. All objects defined below certainly depend on this choice of $p$-adic integer roots, but we suppress that dependence in the subsequent notation. 

\noindent By reducing modulo prime powers and applying the Chinese Remainder Theorem, the choice of roots determines, for each $d\in \N$, a unique $\bsr_d \in (-d,0]^{\ell}$ with $\bsr_d \equiv \bsz_p \ \text{mod }p^j$ for all prime powers $p^j\mid d$.

\noindent We define a completely multiplicative function $\lambda$ (depending on $h$ and $\{\bsz_p\}$) on $\N$ by letting $\lambda(p)=p^{m_p}$ for each prime $p$, where $m_p$ is the multiplicity of $\bsz_p$ as a root of $h$, that is, $$m_p=\min\left\{i_1+\cdots+i_{\l} : \frac{\partial^{i_1+\cdots+i_{\l}}h}{\partial x_1^{i_1}\cdots \partial x_{\l}^{i_{\l}}} (\bsz_p) \neq 0\right\}. $$
Roughly speaking, $\lambda(d)$ is the largest guaranteed factor of $h(\bsn)$ for $\bsn\equiv \bsr_d \ (\text{mod }d)$.

\begin{definition}
With notation as described above, for each $d\in \N$ we define the \textit{auxiliary polynomial} $h_d\in \Z[x_1,\dots,x_{\l}]$ by $$h_d(\bsx)=h(\bsr_d+d\bsx)/\lambda(d). $$

\end{definition}

Combining the hypotheses of Theorem \ref{delmain} with the technical details of the density increment iteration, the following definition captures a sufficient condition for the success of the method. 

\begin{definition}\label{defn:smoothDeligne} When considering polynomials with integer coefficients, we use the terms \textit{smooth} and \textit{Deligne} as previously defined by embedding the coefficients in the field of rational numbers. In particular, $h\in \Z[x_1,\dots,x_{\l}]$ of degree $k\geq 1$ is Deligne if the system $ h^k(\bsx)=\frac{\partial h^k}{\partial x_1}(\bsx)=\cdots=\frac{\partial h^k}{\partial x_{\l}}(\bsx)=0 $ has no solution besides $x_1=\cdots=x_{\l}=0$ in $\bar{\Q}^{\l}$. In this case, there exists a finite set of primes $X=X(h)$ such that the reduction of $h$ modulo $p$ is Deligne for all $p\notin X$: Indeed, one can take $X(h)$ to be the set of primes dividing the Macaulay resultant $\Res\left(h^k, \frac{\partial h^k}{\partial x_1}, \cdots, \frac{\partial h^k}{\partial x_\ell}\right)$, which is nonzero precisely when $h$ is Deligne. (See also Prop. A.9.1.6 of \cite{HindrySilverman}.)

\noindent Further, we say that $h$ is \textit{strongly Deligne} if there exists a finite set of primes $X=X(h)$ and a choice $\{\bsz_p\}_{p\in \P}$ of $p$-adic integer roots of $h$ such that the reduction of $h_d$ modulo $p$ is Deligne for all $p\notin X$ and all $d\in \N$. We note that strongly Deligne polynomials are necessarily both Deligne and intersective.

\noindent To highlight some of the subtlety of this definition, we first note that  $h_d^k=\frac{d^k}{\lambda(d)}h^k$, so for a prime $p\nmid d$, we have that if $h$ is Deligne modulo $p$, then $h_d$ is Deligne modulo $p$. However, complications arise when $p\mid d$, because $h_d^i$ has a factor of $d^i/\lambda(d)$, and hence vanishes modulo $p$ for all $i>m_p$. For an example of a polynomial that is Deligne and intersective but not strongly Deligne, see ``the ugly''  in Section \ref{summsec}.


\end{definition}


 
For $k,\l\geq 2$, we let $\mu(k,\l)=\begin{cases} [(k-1)^2+1]^{-1} & \text{if }\l=2 \\ 1/2 &\text{if }\l \geq 3  \end{cases}. $ The central result of this paper is the following:
  
\begin{theorem} \label{more} If $\ell \geq 2$  and $h\in \Z[x_1,\dots,x_{\l}]$ is a strongly Deligne polynomial of degree $k\geq 2$, then \begin{equation}\label{lvb1} D(h(\Z^{\l}),N) \ll_h N e^{-c(\log N)^{\mu(k,\l)}}, \end{equation} where $c=c(h)>0$.

\end{theorem}  

\begin{rem} In Theorem \ref{more}, the full image $h(\Z^{\l})$ is considered for ease of exposition, and to make the conclusion invariant under input translation. However, by inspection of the proof, the same upper bound can be seen to hold for $D(h([1,N^{\epsilon}]^{\l}),N)$ for any $\epsilon>0$, with $c$ and the implied constant depending on $\epsilon$. Also, in several of our results and definitions, we exclude the case $k=1$ only out of convenience due to its triviality in this context. Specifically, if $h\in \Z[x_1,\dots,x_{\l}]$ with $\deg(h)=1$, then $D(h(\Z^{\l}),N)\ll_h 1$ if $0\in h(\Z^{\l})$ and $D(h(\Z^{\l}),N)\gg_h N$ otherwise.
\end{rem}

After setting the stage with preliminary definitions and observations in Section \ref{prelimsec}, we prove Theorem \ref{more} in Section \ref{itproof}, and then establish the needed exponential sum estimates, which we state separately as Theorem \ref{standalonethm}, in Section \ref{expest}. More imminently, in Sections~\ref{sec:integer_root} and \ref{sec:deligne_case}, we describe sufficient conditions under which $h \in \Z[x_1,\ldots,x_\ell]$ is strongly Deligne, and hence (\ref{lvb1}) holds. Then, in Section~\ref{sec:singular}, we explain that in many cases we may still get a bound similar to (\ref{lvb1}) even when the strongly Deligne condition is significantly relaxed.
\subsection{The integer root case}\label{sec:integer_root} The simplest sufficient condition for the intersectivity of a nonzero polynomial is the existence of an integer root. In this case, all $p$-adic integer roots can be taken to equal said integer root, which simplifies the auxiliary polynomial definition, giving rise to a pleasantly tangible sufficient condition for the strongly Deligne property, as captured with the following proposition.

\begin{proposition}\label{prop:integer_root} Suppose $\ell \geq 2$ and $h\in \Z[x_1,\dots,x_{\l}]$ with $h(\bszero)=0$. If  the highest and lowest degree homogeneous parts of $h$ are smooth, then $h$ is strongly Deligne. 
\end{proposition}

\begin{proof} Suppose $h$ satisfies the hypotheses, let $k=\deg(h)$, let $j$ denote the lowest degree of the nonzero terms of $h$, and let $X$ denote the finite set of primes $p$ such that $p\mid jk$ or either $h^k$ or $h^j$ is not smooth modulo $p$. Making the natural choice of $p$-adic integer roots $\bsz_p=\bszero$ for all $p$, we then have $h_d(\bsx)=h(d\bsx)/d^j$, hence $h^i_d(\bsx)=d^{i-j}h^i(\bsx)$. Fix $p\notin X$. If $p\nmid d$, then the highest degree part of $h_d$ modulo $p$ is a nonzero multiple of $h^k$, which is smooth modulo $p$, hence $h_d$ is Deligne modulo $p$. If $p\mid d$, then the only nonvanishing homogeneous part of $h_d$ is precisely $h^j$, which is smooth modulo $p$, hence $h_d$ is Deligne modulo $p$.\end{proof}
 
\begin{rem}
We note that $h(\Z^{\l})$, hence the threshold $D(h(\Z^{\l}),N)$, as well as the Deligne and strongly Deligne properties, are all invariant under translations of the form $h(\bsx+\bsn)$ for a fixed $\bsn\in \Z^{\l}$. In particular, Proposition \ref{prop:integer_root} applies provided there exists $\bsn\in \Z^{\l}$ such that $h(\bsn)=0$ and the highest and lowest degree parts of $h(\bsx+\bsn)$ are smooth. More generally, all of our results that hold for a polynomial $h$ also hold for the full translation equivalence class of $h$.
\end{rem}
 
For homogeneous bivariate polynomials, smoothness of the corresponding ($0$-dimensional) variety is equivalent to non-vanishing of the discriminant. Therefore, for $\ell = 2$, we have the following, which in particular combines with Theorem \ref{more} to yield Theorem \ref{introthm} as a special case.

\begin{corollary}\label{2vcor} Suppose $h\in \Z[x,y]$ with $h(0,0)=0$. If  the highest and lowest degree homogeneous parts of $h$ have nonzero homogeneous discriminant, then $h$ is strongly Deligne.

\end{corollary}

\subsection{The Deligne case}\label{sec:deligne_case}

Taking the next step in complexity, here we consider the case of a polynomial that is Deligne and intersective, but may not have an integer root. Recalling that if $p\mid d$, then $h^i_d$ vanishes modulo $p$ for all $i>m_p$, we make the following definition with the hopes of exploiting the fact that a nonzero homogeneous linear polynomial is guaranteed to be smooth.

\begin{definition} For $\ell \in \N$ and $h\in \Z[x_1,\ldots,x_{\l}]$ we say that $h$ is \textit{smoothly intersective} if there exists a choice  $\{\bsz_p\}_{p\in \mathcal{P}}$ of $p$-adic integer roots of $h$ such that $m_p=1$ for all but finitely many $p$. In other words, the variety defined by $h=0$ has at least one point over $\Z_p$ for all $p$, and at least one nonsingular point over $\Z_p$ for all but finitely many $p$.
\end{definition}

For low-hanging examples of polynomials that are intersective but not smoothly intersective, one could consider the square of any intersective polynomial, but such polynomials do not pass even our coarsest of nonsingularity filters.  For an example of a polynomial that is intersective but not smoothly intersective in a more subtle and problematic way, see our discussion of ``the ugly'' in Section \ref{summsec}. Combining the motivation for the smoothly intersective definition with the fact that the highest degree part of a Deligne polynomial is assumed to be smooth, the following proposition provides a sufficient condition for the strongly Deligne property, and includes two notable special cases.

\begin{proposition}\label{thm:main} Suppose $\ell \ge 2$ and $h \in \Z[x_1,\ldots,x_{\l}]$ is Deligne and intersective with $\deg(h)=k\geq 2$. If there exists a choice $\{\bsz_p\}_{p\in \P}$ of $p$-adic integer roots of $h$ satisfying $m_p\in \{1,k\}$ for all but finitely many $p$, then $h$ is strongly Deligne. In particular, if $k=2$ or $h$ is smoothly intersective, then $h$ is strongly Deligne.  
 
\end{proposition} 




Using estimates on the number of  nonsingular points on irreducible varieties over finite fields, we obtain the following convenient criterion for smooth intersectivity.

\begin{proposition}\label{prop:main}
Suppose $\ell \ge 2$ and $h \in \Z[x_1,\ldots,x_{\ell}]$ is Deligne and intersective, and let $h=g_1\cdots g_n$ be an irreducible factorization of $h$ in $\Z[x_1,\dots,x_{\l}]$. If $g_i$ is geometrically irreducible for some $1\leq i \leq n$, then $h$ is smoothly intersective, hence strongly Deligne.
\end{proposition}

\begin{rem} 
The conclusion of Proposition~\ref{prop:main} remains true under weaker assumptions on the factorization of $h$. We give this cleaner statement here, but prove the more general statement in Corollary~\ref{prop:main_stronger}.
\end{rem}


%



For $\ell \ge 3$, the Deligne condition actually implies geometric irreducibility, yielding the following:

\begin{corollary}\label{3var} Suppose $\l \geq 3$ and $h\in \Z[x_1,\dots,x_{\l}]$. If $h$ is Deligne and intersective, then $h$ is smoothly intersective, hence strongly Deligne.

\end{corollary} 

\begin{proof}
By Proposition~\ref{prop:main}, it suffices to show that if $h$ is a Deligne polynomial in $\ell \ge 3$ variables, then $h$ is geometrically irreducible. Suppose to the contrary that $h = g_1g_2$ with $g_1,g_2 \in \QQbar[x_1,\ldots,x_\ell]$ nonconstant of degrees $d$ and $k - d$, respectively. In particular, we have $h^k = g_1^dg_2^{k-d}$. Each of $\{g_1^d = 0\}$ and $\{g_2^{k-d} = 0\}$ has codimension $1$ in $\bbP^{\ell -1}$ (since they are hypersurfaces) and dimension at least $1$ (since we assumed $\ell \ge 3$). In particular, $\{g_1^d = 0\}$ and $\{g_2^{k-d} = 0\}$ have nontrivial intersection, and any intersection point must be a singular point of the union $\{h^k = 0\}$, contradicting the fact that $h$ is Deligne.
\end{proof}

In Section \ref{AG1}, we collect some crucial tools from algebraic geometry, which are followed by the proofs of both Proposition \ref{thm:main} and the aforementioned generalization of Proposition \ref{prop:main}.
 
\subsection{The singular case}\label{sec:singular}

While the Deligne condition is required to apply Theorem \ref{delmain} to get the desired cancellation in our exponential sums, brief consideration reveals that the condition is not strictly necessary for a bound like (\ref{lvb1}) to hold, provided the failure of the Deligne condition is in balance with the freedom of extra variables.  For a particularly simple example, consider $h(x,y,z)=(x+z)^4+(x+z)y^3+y^4.$ This is a homogeneous degree-$4$ polynomial, and the variety $\widehat{V}\subseteq \bbP^2$ defined by its vanishing has a unique singular point, namely $(1:0:-1)$. In particular, $h$ is not Deligne. However, by fixing $z=0$, we can define $g(x,y)=h(x,y,0)=x^4+xy^3+y^4$, which is a bivariate homogeneous polynomial of nonzero discriminant. In particular, $g$ is strongly Deligne, so Theorem \ref{more} applies, and moreover $g(\Z^2)=h(\Z^3)$, so (\ref{lvb1}) holds for $h$ as well, applied as if $\l=2$ as opposed to $\l=3$.

This example hints at a less black-and-white consideration of the singularity of a projective variety. For $h\in \Z[x_1,\dots,x_{\l}]$ with $\deg(h)=k\geq 1$, $h$ is Deligne precisely when the variety $\widehat{V}\subseteq \bbP^{\l-1}$ defined by $h^k=0$ is nonsingular. The example above indicates that we should really only need to avoid this variety being ``too singular'', which leads to the following definition.

\begin{definition} For $\ell \in \N$ and a nonconstant homogeneous polynomial $g\in \Z[x_1,\dots,x_{\ell}]$, let $\widehat{V}\subseteq \mathbb{P}^{\ell-1}$ be the variety defined by $g=0$, and let $\widehat{V}^\sing$ be the singular locus of $\widehat{V}$. We define the \textit{rank} of $g$ to be the codimension of $\widehat{V}^\sing$ in $\mathbb{P}^{\ell-1}$, with the convention that the empty set has dimension $-1$, hence the codimension of the empty set in $\bbP^{\l-1}$ is $\l$. This is a notion of rank developed by Birch in \cite{birch} and utilized, for example, in \cite{cookmagyar}.

\end{definition}

For $h\in \Z[x_1,\dots,x_{\l}]$ with $\deg(h)=k\geq 1$, the rank of $h^k$ should, roughly speaking, encode the number of variables $r$ such that $g(\Z^r)\subseteq h(\Z^{\l})$ for some Deligne polynomial $g\in \Z[x_1,\dots,x_{r}]$. In particular, $h$ is Deligne if and only if the rank of $h^k$ is $\l$. In Section \ref{dimlowsec}, using careful dimension-lowering arguments, we successfully expand the class of polynomials for which a result analogous to Theorem \ref{more} holds, generalizing our efforts from Sections \ref{sec:integer_root} and \ref{sec:deligne_case} as follows.
 


\begin{theorem} \label{dimlowrootthm} Suppose $\ell \geq 2$ and $h \in \Z[x_1,\dots,x_{\ell}]$ with $h(\bszero)=0$ and $\deg(h)=k\geq 2$. Let $r$ be the minimum rank of the highest and lowest degree homogeneous parts of $h$. If $r\geq 2$, then \begin{equation}\label{lvb} D(h(\Z^{\ell}),N) \ll_h N e^{-c(\log N)^{\mu(k,r)}},\end{equation} where $c=c(h)>0$.
\end{theorem}

\begin{theorem} \label{dimlowthm} Suppose $\ell \geq 2$ and $h \in \Z[x_1,\dots,x_{\ell}]$ is intersective of degree $k\geq 2$. Let $r$ be the rank of $h^k$. If $r\geq 3$, OR if $r=2$ and there exists a choice $\{\bsz_p\}_{p\in \P}$ of $p$-adic integer roots of $h$ satisfying $m_p\in \{1,k\}$ for all but finitely many $p$, then (\ref{lvb}) holds. 
\end{theorem}

\begin{rem} To shed light on the hypotheses of Theorems \ref{dimlowrootthm} and \ref{dimlowthm}, we note that, for $\ell \geq 2$ and a nonconstant homogeneous polynomial $g\in \Z[x_1,\dots,x_{\l}]$ of rank $r$, we have $r\geq 2$ if and only if $g$ is squarefree---in other words, if and only if $g = 0$ defines a reduced variety.
\end{rem}


\subsection{Summary of results} \label{summsec} For this section, we suppose $k,\l\geq 2$ and $h\in \Z[x_1,\dots,x_{\l}]$ with $\deg(h)=k$, and we let $r$ denote the rank of $h^k$. We assume $h$ is  intersective, as otherwise $D(h(\Z^{\l}),N)\gg_h N$. The following bullet points summarize the reach and limitations of our results.

\begin{itemize} \item \textbf{The good:} In addition to previously known results on sums of univariate intersective polynomials (Theorems 1.2 and 5.7 of \cite{ricemax}), we now have that (\ref{lvb}) holds provided $h$, or in the case of (\ref{rootitem}) any translation of $h$, meets any of the following criteria: 

 \begin{enumerate}[(i)] \item $r\geq 3$ (including Deligne with $\l \geq 3$)  \item \label{quaditem} $r=k=2$ (including Deligne with $\l=k=2$)  \item \label{rootitem} $r=2$ (including Deligne with $\ell=2$), $h(\bszero)=0$, and the lowest degree homogeneous part of $h$ has rank at least $2$. This includes as a special case bivariate homogeneous polynomials with nonzero discriminant, which is Theorem \ref{introthm} from the introduction.   \item \label{smoothitem} $r=2$ (including Deligne with $\ell=2$) and $h$ is smoothly intersective, the latter of which in particular holds if any irreducible (over $\Z$) factor of $h$ is geometrically irreducible. Parts of this item can be made slightly more general, as seen in the hypotheses of Proposition \ref{prop:main} and Corollary \ref{prop:main_stronger}.  

\noindent An interesting example of (\ref{smoothitem}) that does not fit into any other category is $h(x,y)=x^3+y^3-q$, where $q$ is a prime congruent to $1$ modulo $90090$ that is not expressable as the sum of two integer cubes, of which there are plenty. This polynomial has no rational root, and it cannot be decomposed into a sum of two univariate intersective polynomials, but it is Deligne and it has simple roots in $\Z_p^2$ for all primes $p$. This example was discussed in a remark following Theorem 1.2 in \cite{ricemax} to illustrate a limitation of that result. \end{enumerate}

\item \textbf{The bad:} The methods utilized here fail to improve on univariate results in the case that $r=1$, or equivalently the case that $h^k$ has a repeated factor. It should be noted that we can only definitively say that it is impossible to reach beyond the cutting edge of the univariate setting if $h=f\circ g$ for some $g\in \Z[x_1,\dots,x_{\l}]$ and $f\in \Z[x]$ with $\deg(f)\geq 2$, because in this case $h(\Z^{\l})\subseteq f(\Z)$. This was hinted at in the introduction with the example $h(x,y)=(x+y)^2$. In this situation, $h^k$ is a proper power of the highest-degree part of $g$, so we definitely have $r=1$. While it is certainly possible to have $r=1$ without $h$ being given as a composition of this sort, our current methods cannot distinguish between the two.

\item \textbf{The ugly:} A more subtle remaining hurdle is the case where $r=2$ (including Deligne with $\l=2$), $k\geq 3$, and $h$ does not meet either of the criteria described in items  (\ref{rootitem}) or (\ref{smoothitem}). Focusing on the $\l=2$ Deligne case, such a polynomial must satisfy $\Delta(h^k)\neq 0$, must be intersective and hence have roots in $\Z_p^{2}$ for all primes $p$, but by Proposition \ref{thm:main}, for infinitely many $p$, all roots in $\Z_p^{2}$ must satisfy $2\leq m_p\leq k-1$. In particular, by Proposition \ref{prop:main}, at least one coefficient in every geometrically irreducible factor of $h$ must fail to be an integer. Finally, by Corollary \ref{2vcor}, if $h$ satisfies $h(0,0)=0$, then the lowest degree part of $h$ must have discriminant $0$.

\noindent One example is   $$ h(x,y) = x^4 - 2y^4 + 2x^2(x + y) + (x + y)^2 = \left(x^2 - \sqrt{2}y^2 + (x + y)\right)\left(x^2 + \sqrt{2}y^2 + (x + y)\right).$$ 
For any  prime $p$ such that $2$ is not a square in $\Q_p$, the only $\Q_p$ roots of $h$ are $(0,0)$ and $(-1,0)$. With these choices for $\bsz_p$, the highest degree nonvanishing part of $h_p$ modulo $p$ is either $(x+y)^2$ or $(x - y)^2$, respectively. In both cases $\Delta(h_p^2)=0$, and hence $h_p$ is not Deligne at this infinite collection of primes. In other words, $h$ is not strongly Deligne, and we cannot claim that (\ref{lvb}) holds.
 
\item \textbf{The future:} The issue in the previous bullet point may represent an avoidable artifact of the method, in which case the upper bound (\ref{lvb}) could be shown to hold for all intersective polynomials satisfying $r\geq 2$. Regarding improved bounds, as noted in Section 2.3 of \cite{Ricebin}, and as implicitly referenced in \cite{Ruz} when noting that the exponent $\mu$ in (\ref{rsb}) could be increased to $1/2$ conditioned on the Generalized Riemann Hypothesis, an upper bound of order $Ne^{-c\sqrt{\log N}}$ appears to be the limit of a Fourier analytic $L^2$ density increment. More specifically, if $(\delta,N)\mapsto (\delta',N')$ represents the change in density and interval size at each step of the iteration, then any further improvement would require either $N'/N$ to decay more slowly than any power of $\delta$, or $\delta'/\delta$ to tend to infinity, as $\delta\to 0$, both of which appear incompatible with the method.  To be clear, this is not at all to say that much stronger upper bounds do not hold, even in the univariate polynomial setting. As discussed in Section \ref{lbsec}, this question is rather murky. However, to achieve such a goal would likely require a fundamentally different proof strategy. 
\end{itemize}






\section{Preliminaries} \label{prelimsec}

In this section we make some preliminary definitions and observations required to execute the sieve-powered $L^2$ density increment strategy utilized to prove Theorem \ref{more}.

\subsection{Fourier analysis and the circle method on $\Z$} We embed our finite sets in $\Z$, on which we utilize an unnormalized discrete Fourier transform. Specifically, for a function $F: \Z \to \C$ with finite support, we define $\widehat{F}: \T \to \C$, where $\T$ denotes the circle  parameterized by the interval $[0,1]$ with $0$ and $1$ identified, by \begin{equation*} \widehat{F}(\alpha) = \sum_{x \in \Z} F(x)e^{-2 \pi ix\alpha}. \end{equation*}
\noindent Given $N\in \N$ and a set $A\subseteq [1,N]$ with $|A|=\delta N$, we examine the Fourier analytic behavior of $A$ by considering the \textit{balanced function}, $f_A$, defined by
$ f_A=1_A-\delta 1_{[1,N]}.$

\noindent As is standard, we decompose the frequency space into two pieces: the points of $\T$ that are close to rational numbers with small denominator, and the complement.

\begin{definition}Given $\gamma>0$ and $Q\geq 1$, we define, for each $a,q\in \N$ with $0\leq a \leq q-1$,
$$\mathbf{M}_{a/q}(\gamma)=\left\{ \alpha \in \T : \Big|\alpha-\frac{a}{q}\Big| < \gamma \right\},  \ \mathbf{M}_q(\gamma)=\bigcup_{(a,q)=1} \mathbf{M}_{a/q}(\gamma), \text{ and }\mathbf{M}'_q(\gamma)=\bigcup_{r\mid q} \mathbf{M}_r(\gamma)=\bigcup_{a=0}^{q-1} \mathbf{M}_{a/q}(\gamma).$$  We then define the \textit{major arcs} by
$$ \mathfrak{M}(\gamma,Q)=\bigcup_{q=1}^{Q} \mathbf{M}_q(\gamma),$$ 
and the \textit{minor arcs} by  
$\mathfrak{m}(\gamma,Q)=\T\setminus \mathfrak{M}(\gamma,Q).
$ 
We note that if $2\gamma Q^2<1$, then \begin{equation} \label{majdisj}\mathbf{M}_{a/q}(\gamma)\cap\mathbf{M}_{b/r}(\gamma)=\emptyset \end{equation}whenever  $a/q\neq b/r$ and  $q,r \leq Q$. 
\end{definition} 

\subsection{Inheritance proposition} As previously noted, we defined auxiliary polynomials to keep track of an inherited lack of prescribed differences at each step of a density increment iteration. The following proposition makes this inheritance precise.

\begin{proposition} \label{inh} Suppose $\l\in \N$, $h\in\Z[x_1,\dots,x_{\ell}]$ is intersective, $d,q\in \N$, and $A\subseteq \N$.

\noindent If $(A-A)\cap h_d(\Z^{\ell})\subseteq \{0\}$ and $A'\subseteq \{a: x+\lambda(q)a \in A\}$ for some $x\in \Z$, then $(A'-A')\cap h_{qd}(\Z^{\ell})\subseteq \{0\}$.
\end{proposition}
\begin{proof} Suppose that $A\subseteq \N$, $A'\subseteq \{a : x+\lambda(q)a \in A\}$, and $$a-a'=h_{qd}(\bsn)=h(\bsr_{qd}+qd\bsn)/\lambda(qd)\neq 0$$ for some $\bsn\in \Z^{\ell}$, $a, a' \in A'$. By construction we have that $\bsr_{qd}\equiv \bsr_d$ mod $d$, so there exists $\boldsymbol{s}\in \Z^{\ell}$ such that $\bsr_{qd}=\bsr_d+d\boldsymbol{s}$. Further, $\lambda$ is completely multiplicative, and therefore $$0\neq h_d(\boldsymbol{s}+q\bsn)=h(\bsr_d+d(\boldsymbol{s}+q\bsn))/\lambda(d)=\lambda(q)h_{qd}(\bsn)=\lambda(q)a-\lambda(q)a' \in A-A.$$ Since $a-a'\neq 0$, we have $(A-A)\cap h_d(\Z^{\ell})\not\subseteq \{0\}$, and the contrapositive is established. 
\end{proof} 

\subsection{Sieve definitions and observations} \label{sievesec} As in \cite{ricemax}, we apply a polynomial-specific sieve to our set of considered inputs in order to, roughly speaking, reduce our analysis of local exponential averages to the case of prime moduli, which in the multivariate setting allows for the application of Theorem \ref{delmain}. To this end, for $\l\in \N$, an intersective polynomial $h\in\Z[x_1,\dots,x_{\ell}]$, and each prime $p$ and $d\in \N$, we define $\gamma_{d}(p)$ to be the smallest power such that $\grad h_d$ modulo $p^{\gamma_{d}(p)}$ does not vanish identically {\it as a function on $(\Z/p^{\gamma_d(p)}\Z)^{\l}$}, and we let $j_d(p)$ denote the number of solutions to $\grad h_d = \bszero$ in $(\Z/p^{\gamma_{d}(p)}\Z)^{\l}$. Then, for $d\in \N$ and $Y>0$ we define $$W_d(Y)=\left\{\bsn\in \N^{\l}: \grad h_d(\bsn) \not\equiv \bszero \text{ mod } p^{\gamma_{d}(p)} \text{ for all } p\leq Y \right\}.$$
In the absence of a subscript $d$ in the usage of $\gamma(p), j(p),$ and $W(Y)$, we assume $d=1$, in which case the definitions make sense even for non-intersective polynomials. Further, for any $g\in \Z[x_1,\dots,x_{\ell}]$ and $q\in \N$, we define $$W^{q}(Y)=\left\{\bsn\in \N^{\l}: \grad g(\bsn) \not\equiv \bszero \text{ mod }p^{\gamma(p)} \text{ for all } p\leq Y, \ p^{\gamma(p)}\mid q \right\}.$$ 
Unlike in the univariate case, the size of $W(Y)$ here can be estimated with a straightforward application of the inclusion-exclusion principle, as opposed to a Brun sieve truncation thereof (see Proposition 2.4 in \cite{ricemax}). To achieve this goal, however, we must first look forward and invoke an estimate established in Section \ref{gv2}. For the following two statements, we assume $\ell \geq 2$ and $g\in \Z[x_1,\dots,x_{\l}]$ with $\deg(g)=k\geq 1$.

\begin{lemma} \label{gradconst} If $p$ is prime and $g$ is Deligne modulo $p$, then $j(p)\ll_{k,\l} 1$.

\end{lemma} 

\begin{proposition}\label{brunprop} For any $x_1,\dots,x_{\l},Y>0$ we have \begin{equation}\label{sieve} \left|B\cap W(Y)\right| = x_1x_2\cdots x_{\l}\prod_{p\leq Y} \left(1-\frac{j(p)}{p^{\gamma(p)\ell}} \right)+E, \end{equation} where $B=[1,x_1]\times\cdots \times [1,x_{\l}]$, $$E=\begin{cases} O(X^{\l-1}\log^C(Y)) & \text{if }\l=2 \\ O(X^{\l-1})& \text{if }\l \geq 3 \end{cases}, $$ $X=\max\{x_1,\dots,x_{\l}\}$, $C=C(k,\l)$, and the implied constants depend only on $k$, $\l$, the moduli at which $\grad g$ identically vanishes, and the primes $p\leq Y$ modulo which $g$ is not Deligne.
 
\end{proposition}

\begin{proof} Fix $x_1,\dots,x_{\l},Y>0$ and let $X=\max\{x_1,\dots,x_{\l}\}$. For primes $p_1<p_2<\cdots<p_s$, we let \begin{equation*}\mathcal{A}_{p_1\cdots p_s}=\mathcal{A}_{p_1\cdots p_s}(x_1,\dots,x_{\l}) = \left|\left\{\bsn \in B: \grad g (\bsn) \equiv \bszero\text{ mod } p_i^{\gamma(p_i)} \ \text{for all } 1\leq i \leq s  \right\}\right|.\end{equation*} 
 
\noindent Fixing $Y>0$ and letting $r$ denote the number of primes that are at most $Y$, we have by the Chinese Remainder Theorem and the inclusion-exclusion principle that \begin{equation}\label{brunalt1} \left|B \cap  W(Y)\right|= \sum_{s=0}^r  (-1)^s \sum_{p_1<\dots < p_s\leq Y} \mathcal{A}_{p_1\cdots p_s}.\end{equation} Further,  \begin{equation}\label{AP1} \mathcal{A}_p = \frac{j(p)}{p^{\gamma(p)\ell}}x_1\cdots x_{\l} + R_p, \end{equation} where $|R_p|\ll_{\l} j(p)(X/p^{\gamma(p)})^{\ell-1}$.  We trivially have $j(p) \leq p^{\gamma(p)\ell}$, while if $g$ is Deligne modulo $p$, then $\gamma(p)=1$ and, by Lemma \ref{gradconst}, $j(p)\leq C=C(k,\l)$. In particular, we can apply the Chinese Remainder Theorem again and extend (\ref{AP1}) to \begin{equation} \label{AP2}  \mathcal{A}_{p_1\cdots p_s} = x_1\cdots x_{\l} \prod_{i=1}^s\frac{j(p_i)}{p_i^{\gamma(p_i)\ell}} + R_{p_1\cdots p_s}, \end{equation} where $|R_{p_1\cdots p_s}| \leq KC^s(X/p_1\cdots p_s)^{\l-1}$, where $K$ depends only on the moduli at which $\grad g$ identically vanishes and the primes $p\leq Y$ modulo which $g$ is not Deligne.  Now, by (\ref{brunalt1}) and (\ref{AP2}) we have 
\begin{align*}  \left|B \cap  W(Y)\right|&= \sum_{s=0}^r  (-1)^s \sum_{p_1<\dots < p_s\leq Y} \mathcal{A}_{p_1\cdots p_s}\\ &=\sum_{s=0}^r  (-1)^s \sum_{p_1<\dots < p_s\leq Y}\left(x_1\cdots x_{\l} \prod_{i=1}^s\frac{j(p_i)}{p_i^{\gamma(p_i)\ell}} + R_{p_1\cdots p_s}\right) \\ &= x_1x_2\cdots x_{\l}\prod_{p\leq Y} \left(1-\frac{j(p)}{p^{\gamma(p)\ell}} \right)+E, \end{align*} where \begin{align*}|E| &\leq KX^{\l-1}\sum_{s=0}^r \sum_{p_1<\dots < p_s\leq Y} \frac{C^s}{(p_1\cdots p_s)^{\l-1}} = KX^{\l-1}\prod_{p\leq Y}\left(1+\frac{C}{p^{\l-1}}\right),
	\end{align*} and the estimate follows. \end{proof} 

\subsection{Control over gradient vanishing: Part I} A potential hazard of the density increment method is the possibility that, as $d$ grows, $\grad h_d$ could identically vanish at a larger and larger collection of moduli. This section is dedicated to establishing that, for strongly Deligne polynomials, this does not occur. We begin by noting that the collection of moduli at which a polynomial identically vanishes is firmly controlled in terms of its degree and the gcd of its coefficients. Throughout this section we assume $k,\l\in \N$.

\begin{definition} We define a \textit{multi-index} to be an $\ell$-tuple $\bsi=(i_1,\dots,i_{\ell})$ of nonnegative integers. We let $|\bsi|=i_1+\dots+i_{\ell}$, we let $\bsi ! =i_1!\cdots i_{\l}!$, and for $\bsx=(x_1,\dots,x_{\ell})$, we let $\bsx^{\bsi}=x_1^{i_1}\cdots x_{\ell}^{i_{\ell}}$. Finally, for a polynomial $g(\bsx)$, we let $\partial^{\bsi}g=\frac{\partial^{|\bsi|} g}{\partial x_1^{i_1}\cdots \partial x_{\ell}^{i_{\ell}}}$.

\end{definition}

\begin{proposition} \label{idzero} If  $g(\bsx)= \sum_{|\boldsymbol{i}|\leq k}a_{\boldsymbol{i}}\bsx^{\boldsymbol{i}}  \in \Z[x_1,\dots,x_{\l}]$ is identically zero modulo $q\in \N$, then $$q \mid k!\gcd(\{a_{\bsi}\}). $$
\end{proposition} 

\begin{proof} We first note that $g$ is identically zero as a function on $\Z/q\Z$ if and only if the polynomial $g/q$ is integer-valued. In this case, since products of binomial coefficients $$\binom{\bsx}{\bsi}=\binom{x_1}{i_1}\cdots \binom{x_{\l}}{i_{\l}}=\frac{x(x-1)\dots(x-i_1+1)}{i_1!} \cdots \frac{x(x-1)\dots(x-i_{\l}+1)}{i_{\l}!}$$form a $\Z$-basis for integer-valued polynomials in $\Q[x_1,\dots,x_{\l}]$, we can write $g(x)=\sum_{|\bsi|\leq k}qb_{\bsi}\binom{\bsx}{\bsi}$ for $b_{\bsi} \in \Z$. In particular, by clearing denominators we see that the coefficients of $k!g$ are all divisible by $q$, and the proposition follows.
\end{proof}

\noindent Further, we note that the gcd of the coefficients of each partial derivative of a polynomial $h\in \Z[x_1,\dots,x_{\l}]$ divides $k!$ times the gcd of the nonconstant coefficients of $h$. With this in mind, the following definition and proposition complete the task at hand. 

\begin{definition} For    $h(\bsx)= \sum_{|\boldsymbol{i}|\leq k}a_{\boldsymbol{i}}\bsx^{\boldsymbol{i}}  \in \Z[x_1,\dots,x_{\l}]$, we define $$\cont(h)=\gcd(\{a_{\boldsymbol{i}}: |\boldsymbol{i}|>0\}).$$ 

\noindent We note that our use of $\cont(h)$ does not precisely align with the standard notion of the \textit{content} of a polynomial, as we exclude the constant coefficient.  

\end{definition}


\begin{proposition}\label{content} If $h\in \Z[x_1,\dots,x_{\l}]$ is a strongly Deligne polynomial of degree $k$, then  $$\cont(h_d) \ll_h 1. $$

\end{proposition}
 
\begin{proof} Suppose  $d\in \N$ and $h\in \Z[x_1,\dots,x_{\l}]$ is a strongly Deligne polynomial of degree $k$. Let $\{\bsz_p\}_{p\in \P}$ and $X$ denote the choice of $p$-adic integer roots and the finite set of primes, respectively, guaranteed by the strongly Deligne condition. In particular, $h_d$ is Deligne modulo $p$ for all $p\notin X$. Because constant polynomials are not Deligne, $\cont(h_d)$ can only be divisible by primes in $X$.

\noindent Recalling that $h_d(\bsx)=h(\bsr_d+d\bsx)/\lambda(d)$, we make the trivial note that for any multi-index $\bsi$ with $|\bsi|=k$, the $\bsx^{\bsi}$ coefficient of $h_d$ is precisely $d^k/\lambda(d)$ times the corresponding coefficient $a_{\bsi}$ of $h$. In particular, \begin{equation} \label{dk} \cont(h_d) \mid \frac{d^k}{\lambda(d)}a_{\bsi} \ \text{whenever} \ |\bsi|=k.
\end{equation}

\noindent Now fix $p\in X$. By definition of the multiplicity $m_p$, there exists a multi-index $\bsi$ with $|\bsi|=m_p$ and $\partial^{\bsi}h(\bsz_p)\neq 0$, so in particular $\partial^{\bsi}h(\bsz_p)$ has some finite $p$-adic valuation $v_1(p)$. 

\noindent If $p^{v_1(p)+1}\nmid d$, then by (\ref{dk}), we have that $p^{kv_1(p)+v_2(p)+1}\nmid \cont(h_d)$, where $v_2(p)$ is the minimum $p$-adic valuation amongst the degree-$k$ coefficients of $h$. Now suppose that $p^{v_1(p)+1}\mid d$. 
 
\noindent Let $b_{\bsi}$ denote the $\bsx^{\bsi}$ coefficient of $h_d$. By Taylor's formula, we have that $$b_{\bsi}=\frac{d^{m_p}}{\lambda(d)}\frac{\partial^{\bsi}h(\bsr_{d})}{\bsi !}.$$ By definition of $\lambda$ we have $p\nmid (d^{m_p}/\lambda(d))$, and since $\bsr_d\equiv \bsz_p \ \text{mod }p^{v_1(p)+1}$ and $p^{v_1(p)+1}\nmid\partial^{\bsi}h(\bsz_{p})$, we have that $p^{v_1(p)+1}\nmid b_{\bsi}$. In either case, we have that $p^{kv_1(p)+v_2(p)+1}\nmid \cont(h_d)$, and hence $$\cont(h_d)\leq \prod_{p\in X}p^{kv_1(p)+v_2(p)+1} \ll_h 1,  $$ as required.\end{proof}

\noindent For strongly Deligne $h\in\Z[x_1,\dots,x_{\ell}]$ with $\deg(h)=k$, we have now established control over not only the error term in the size of $W_d(Y)$, but also the main term, since Lemma \ref{gradconst}, Proposition \ref{idzero}, and Proposition \ref{content}  give \begin{equation}\label{logk} \prod_{p\leq Y} \left(1-\frac{j_d(p)}{p^{\gamma_d(p)\ell}} \right) \gg_h \prod_{C=C(h)\leq p\leq Y} \left(1-\frac{C}{p^{\l}} \right) \gg_h 1\end{equation} for all $d \in \N$ and $Y\geq 2$. 

\subsection{Summary of new exponential sum estimates} \label{standalone} In Section \ref{expest}, we combine new and old techniques to establish the sieved multivariate exponential sum estimates necessary to prove Theorem \ref{more}. These estimates are obtained through a sequence of lemmas presented in the context of the larger proof, so we separately present a summary here in case the estimates are of independent interest to the reader.

For the following theorem, a multivariate generalization of Theorem 2.7 in \cite{ricemax}, we utilize all the sieve-related notation and definitions from Section \ref{sievesec}. Further, we use $\tau$ and $\omega$ to denote the divisor and distinct prime divisor counting functions, respectively, as well as $\phi$ to denote the Euler totient function.

\begin{theorem}\label{standalonethm} For $k,\ell \geq 2$,  $g(\bsx)=\sum_{|\bsi|\leq k} a_{\bsi} \bsx^{\bsi} \in \Z[x_1,\dots,x_{\ell}]$, $J=\sum_{|\bsi|\leq k} |a_{\bsi}|$, and $a,q\in \N$, the following estimates hold:

\begin{enumerate}[(i)]  \item \label{majitem} \textnormal{\textbf{Major arc estimate:}} If $X,Y > 0$ and  $\alpha=a/q+\beta$, then \begin{align*}\sum_{\bsn \in [1,X]^{\ell} \cap W(Y)}e^{2\pi i g(\bsn)\alpha}&=q^{-\ell} \prod_{\substack{ p\leq Y \\ p^{\gamma(p)}\nmid q}}\left(1-\frac{j(p)}{p^{\gamma(p)\ell}}
\right)\sum_{\boldsymbol{s}\in \{0,\dots,q-1\}^{\ell} \cap W^{q}(Y)}e^{2\pi i g(\boldsymbol{s})a/q}\int_{[0,X]^{\ell}}e^{2\pi i g(\bsx)\beta}d\bsx\\\\&\qquad + O_{k,\l}\left(qE(1+JX^{k}|\beta|)^{\ell}\right)  ,\end{align*} where $E$ is as in Proposition \ref{brunprop}. \item \label{locitem} \textnormal{\textbf{Local cancellation:}}  If $(a,q)=1$ and $Y>0$, then $$\left| \sum_{\boldsymbol{s}\in \{0,\dots,q-1\}^{\ell} \cap W^{q}(Y)}e^{2\pi i g(\boldsymbol{s})a/q} \right| \leq C_1\begin{cases} (k-1)^{\l \omega(q)}\Phi(q,\l) q^{\ell/2} &\text{if }q\leq Y \\ C_2^{\omega(q)}\tau(q)^{\l}q^{\ell-1/k} &\text{for all }q \end{cases},$$ where $C_2=C_2(k)$, $\Phi(q,2)=(q/\phi(q))^{C_2}$, $\Phi(q,\l)\ll_{k,\l} 1$ for $\l\geq 3$, and $C_1$ depends only on the moduli at which $\grad g$ identically vanishes and the primes $p\leq Y$ dividing $q$ modulo which $g$ is not Deligne.  \item \label{minitem} \textnormal{\textbf{Minor arc estimate:}} If $X,Y,Z\geq 2$, $YZ\leq X$, $(a,q)=1$, and $|\alpha-a/q|<q^{-2}$, then
\end{enumerate}
$$\left|\sum_{\bsn \in [1,X]^{\ell} \cap W(Y)} e^{2\pi \i g(\bsn)\alpha} \right| \ll_{k,\l} \textnormal{cont}(g)^6(\log Y)^{ek} X^{\ell}\left(e^{-\frac{\log Z}{\log Y}}+\left(J\log^{k^2}(JqX)\left(q^{-1}+\frac{Z}{X}+\frac{qZ^k}{X^k}\right) \right)^{2^{-k}} \right). $$

\end{theorem}

\section{Proof of Theorem \ref{more}} \label{itproof}
In this section, we exploit the estimates enumerated in Theorem \ref{standalonethm} and apply a Fourier analytic $L^2$ density increment, essentially an improved, streamlined version of S\'ark\"ozy's \cite{Sark1} original method, in order to prove Theorem \ref{more}. The core of this method has been utilized in \cite{Lucier}, \cite{LM}, \cite{Ruz}, and \cite{Rice}, among others. Most specifically, this section very closely follows Section 5 of \cite{ricemax}.

\subsection{Main iteration lemma and proof of Theorem \ref{more}}

For the remainder of Section \ref{itproof} we fix $k,\l\geq 2$, a strongly Deligne polynomial $h\in \Z[x_1,\dots,x_{\ell}]$ with $\deg(h)=k$,  and positive constants $C_0=C_0(h)$ and $c_0=c_0(h)$ that are appropriately large and small, respectively. For $N\in \N$ we let $$\cQ=\cQ(N,h)=e^{c_0\sqrt{\log N}}.$$   For a density $\delta\in (0,1]$, we define $\theta(k,\ell,\delta)$ by $\theta(k,\ell,\delta)=1$ if $\ell\geq 3$ and $\theta(k,2,\delta)=\log^{-k(k-2)}((c_0\delta)^{-1})$.

 We deduce Theorem \ref{more} from the following iteration lemma, which makes precise the aforementioned passage from a set lacking nonzero differences in the image of a polynomial to a new, denser subset of a slightly smaller interval lacking nonzero differences in the image of an appropriate auxiliary polynomial. 

\begin{lemma} \label{mainit} Suppose $A\subseteq [1,N]$ with $|A|=\delta N$. If $(A-A)\cap h_d(\Z^{\ell})\subseteq \{0\}$, $C_0,\delta^{-1}\leq \cQ$, and $d\leq N^{c_0}$, then there exists $q\ll_{h} \delta^{-2}$ and $A'\subseteq [1,N']$ such that 
$N'\gg_{h}  \delta^{4k}N$, 
\begin{equation*}|A'|\geq (1+c\theta(k,\ell,\delta))\delta N'   , \end{equation*} where $c=c(h)>0$, and $$(A'-A')\cap h_{qd}(\Z^{\ell})\subseteq \{0\}. $$

\end{lemma}

\begin{proof}[Proof of Theorem~\ref{more}]
Throughout this proof, we let $C$ and $c$ denote sufficiently large or small positive constants, respectively, which we allow to change from line to line, but which depend only on  $h$.  

\noindent Suppose $A \subseteq [1,N]$ with $|A|=\delta N$ and $$(A-A)\cap h(\Z^{\ell})\subseteq \{0\}.$$ Setting $A_0=A$, $N_0=N$, $d_0=1$, and $\delta_0=\delta$, Lemma \ref{mainit} yields, for each $m$, a set $A_m \subseteq [1,N_m]$ with $|A_m|=\delta_mN_m$ and $(A_m-A_m)\cap h_{d_m}(\Z^{\ell})\subseteq \{0\}.$ Further, we have that
\begin{equation} \label{NmI} N_m \geq c\delta^{4k}N_{m-1} \geq (c\delta)^{4km} N,
\end{equation}
\begin{equation} \label{incsizeI} \delta_m \geq (1+c\theta(k,\ell,\delta))\delta_{m-1}, 
\end{equation}
and
\begin{equation}\label{dmI} d_m \leq (c\delta)^{-2} d_{m-1} \leq (c\delta)^{-2 m}, 
\end{equation}
as long as 
\begin{equation} \label{delmI} C,  \delta_m^{-1} \leq e^{c\sqrt{\log N_m}}, \quad d\leq N_m^{c} .
\end{equation}
By (\ref{incsizeI}), the density $\delta_m$ will exceed $1$, and hence (\ref{delmI}) must fail, for $m=M=M(h,\delta)$, where $$M(h,\delta)=\begin{cases} C\log (C\delta^{-1})& \text{if }\ell \geq 3 \\ C\log^{(k-1)^2}(C\delta^{-1}) & \text{if }\ell=2 \end{cases}. $$ However, by (\ref{NmI}), (\ref{incsizeI}), and (\ref{dmI}), (\ref{delmI}) holds for $m=M$ if \begin{equation}\label{endgame} (c\delta)^{4kM}=e^{C\log^{\mu(k,\l)^{-1}}(C\delta^{-1})}\leq N^{c}.\end{equation} Therefore, (\ref{endgame}) must fail, or in other words $\delta \ll_{h} e^{-c(\log N)^{\mu(k,\l)}},$ as claimed. 
\end{proof}

\subsection{$L^2$ Fourier concentration and proof of Lemma \ref{mainit}} The philosophy behind the proof of Lemma \ref{mainit} is that the condition $(A-A)\cap h_d(\Z^{\l})\subseteq \{0\}$ represents highly nonrandom behavior, which should be detectable in the Fourier analytic behavior of $A$. Specifically, we locate one small denominator $q$ such that $\widehat{f_A}$ has $L^2$ concentration around rationals with denominator $q$, then invoke a standard lemma stating that $L^2$ concentration of  $\widehat{f_A}$ implies the existence a long arithmetic progression on which $A$ has increased density. 

\begin{lemma}  \label{L2I} Suppose $A\subseteq [1,N]$ with $|A|=\delta N$, $\eta=c_0\delta$, and  $\gamma=\eta^{-2k}/N$. If $(A-A)\cap h_d(\Z^{\l})\subseteq \{0\}$, $C_0,\delta^{-1}\leq \cQ$, $d\leq N^{c_0}$, and  $|A\cap(N/9,8N/9)|\geq 3\delta N/4$, then there exists $q\leq \eta^{-2}$ such that 
\begin{equation*} \int_{\mathbf{M}'_q(\gamma)} |\widehat{f_A}(\alpha)|^2d\alpha \gg_{h} \theta(k,\ell,\delta) \delta^{2} N.  
\end{equation*}  
\end{lemma}

\noindent Lemma \ref{mainit} follows from Lemma \ref{L2I} and the following standard $L^2$ density increment lemma.

\begin{lemma}[Lemma 2.3 in \cite{thesis}, see also \cite{Lucier}, \cite{Ruz}] \label{dinc} Suppose $A \subseteq [1,N]$ with $|A|=\delta N$. If  $0< \theta \leq 1$, $q \in \N$, $\gamma>0$, and
\begin{equation*} \int_{\mathbf{M}'_q(\gamma)}|\widehat{f_A}(\alpha)|^2d\alpha \geq \theta\delta^2 N,
\end{equation*} 
then there exists an arithmetic progression 
\begin{equation*}P=\{x+\ell q : 1\leq \ell \leq L\}
\end{equation*}
with $qL \gg \min\{\theta N, \gamma^{-1}\}  $ and $|A\cap P| \geq (1+\theta/32)\delta L$.
\end{lemma}

\begin{proof}[Proof of Lemma~\ref{mainit}]
Suppose $A\subseteq [1,N]$, $|A|=\delta N$, $(A-A)\cap h_d(\Z^{\ell})\subseteq \{0\}$, $C_0, \delta^{-1}\leq \cQ$, and $d\leq N^{c_0}$. If $|A\cap (N/9,8N/9)| < 3\delta N/4$, then $\max \{ |A\cap[1,N/9]|, |A\cap [8N/9,N]| \} > \delta N/8$. In other words, $A$ has density at least $9\delta/8$ on one of these intervals. 

\noindent Otherwise, Lemmas \ref{L2I} and \ref{dinc} apply, so in either case, letting $\eta=c_0\delta$, there exists $q\leq \eta^{-2}$ and an arithmetic progression 
\begin{equation*}P=\{x+\ell q : 1\leq \ell \leq L\}
\end{equation*}
with $qL\gg_{h} \delta^{2k} N$ and $$|A\cap P| \geq (1+c\theta(k,\ell,\delta))\delta L  .$$ Partitioning $P$ into subprogressions of step size $\lambda(q)$, the pigeonhole principle yields a progression 
\begin{equation*} P'=\{y+a \lambda(q) : 1\leq a \leq N'\} \subseteq P
\end{equation*}
with $N'\geq qL/2\lambda(q)$ and $|A\cap P'|/N' \geq |A\cap P|/L$. This allows us to define a set $A' \subseteq [1,N']$ by \begin{equation*} A' = \{a \in [1,N'] : y+a \lambda(q) \in A \},
\end{equation*} which satisfies $|A'|=|A\cap P'|$ and $N'\gg_{h} \delta^{2k}N/\lambda(q) \gg_{h} \delta^{4k}N$. Moreover, $(A-A)\cap h_d(\Z^{\ell}) \subseteq \{0\}$ implies $(A'-A')\cap h_{qd}(\Z^{\ell})\subseteq \{0\}$ by Proposition \ref{inh}. 
\end{proof}

Our task for this section is now completely reduced to a proof of Lemma \ref{L2I}.

\subsection{Preliminary notation for proof of Lemma \ref{L2I}} Before delving into the proof of Lemma \ref{L2I}, we take the opportunity to define some relevant sets and quantities, depending on our strongly Deligne polynomial $h\in \Z[x_1,\dots,x_{\ell}]$, scaling parameter $d$, a parameter $Y>0$, and the size $N$ of the ambient interval. In all the notation defined below, we suppress all of the aforementioned dependence, as the relevant objects will be fixed in context.  

We define $W_d$, $\gamma_d$, and $j_d$ in terms of $h$ as in Section \ref{sievesec}. We then let $M=\left(\frac{N}{9J}\right)^{1/k}$,  where $J$ is the sum of the absolute value of all the coefficients of $h_d$, and hence $h_d([1,M]^{\ell})\subseteq [-N/9,N/9]$. We let $$w=\prod_{p\leq Y}\left(1-\frac{j_{d}(p)}{p^{\gamma_{d}(p)\ell}}\right), $$ and we let $T=wM^{\ell}$. 

We let $Z=\{\bsn\in \Z^{\ell}: h_d(\bsn)=0\}$, and we let $H=\left([1,M]^{\ell}\cap W_{d}(Y)\right)\setminus Z$. We note that the hypothesis $\cQ\geq C_0$ allows us to assume at any point that $\cQ$, and hence also $N$, are sufficiently large with respect to $h$, which we take as a perpetual assumption moving forward. Under this assumption, it follows from (\ref{sieve}), (\ref{logk}), and the estimate \begin{equation} \label{Zest} |Z\cap[1,M]^{\ell}|\ll_h M^{\ell-1}\end{equation} that \begin{equation}\label{Hsize} \left| H \right|\geq T/2. \end{equation}

\subsection{Proof of Lemma \ref{L2I}} \label{massproof} Suppose $A\subseteq [1,N]$ with $|A|=\delta N$, $(A-A)\cap h_d(\Z^{\ell}) \subseteq \{0\}$, $C_0,\delta^{-1}\leq \cQ$, and $d\leq N^{c_0}$.  Further, let $\eta=c_0\delta$, let $Q=\eta^{-2}$, and let $Y=\eta^{-2k}$.  Since $h_d(H) \subseteq [-N/9,N/9]\setminus \{0\}$, we have
\begin{align*} \sum_{\substack{x \in \Z \\ \bsn\in H}} f_A(x)f_A(x+h_d(\bsn))&=\sum_{\substack{x \in \Z \\ \bsn \in H}} 1_A(x)1_A(x+h_d(\bsn)) -\delta\sum_{\substack{x \in \Z \\ \bsn\in H}} 1_A(x)1_{[1,N]}(x+h_d(\bsn)) \\ &\qquad -\delta \sum_{\substack{x \in \Z \\ \bsn\in H}} 1_{A}(x+h_d(\bsn))1_{[1,N]}(x)+\delta^2\sum_{\substack{x \in \Z \\ \bsn\in H}} 1_{[1,N]}(x)1_{[1,N]}(x+h_d(\bsn))  \\&\leq \Big(\delta^2N -2\delta|A\cap (N/9,8N/9)|\Big)|H|. 
\end{align*}
Therefore, if $|A \cap (N/9, 8N/9)| \geq 3\delta N/4$, then by (\ref{Hsize}) we have
\begin{equation}\label{neg} \sum_{\substack{x \in \Z \\ \bsn\in H}} f_A(x)f_A(x+h_d(\bsn)) \leq -\delta^2NT/4.
\end{equation} 
We see from (\ref{Zest}) and orthogonality of characters that 
\begin{equation}\label{orth} 
\sum_{\substack{x \in \Z \\ \bsn\in H}} f_A(x)f_A(x+h_d(\bsn))=\int_0^1 |\widehat{f_A}(\alpha)|^2 S(\alpha)d\alpha +O_h(NM^{\ell-1}),
\end{equation} 
where 
\begin{equation*}S(\alpha)= \sum_{\bsn \in [1,M]^{\ell} \cap W_{d}(Y)}e^{2\pi i h_d(\bsn)\alpha}.
\end{equation*} 
Combining (\ref{neg}) and (\ref{orth}), we have  
\begin{equation} \label{mass}
\int_0^1 |\widehat{f_A}(\alpha)|^2|S(\alpha)|d\alpha \geq \delta^2NT/8.
\end{equation} Letting $\gamma=\eta^{-2k}/N$, Theorem \ref{standalonethm} yields that for $\alpha \in \mathbf{M}_q(\gamma), \ q\leq Q $, we have \begin{equation} \label{SmajII} |S(\alpha)| \ll_{h} (k-1)^{\l\omega(q)}\Phi(q,\l)q^{-\ell/2}T,
\end{equation} where $\Phi(q,2)=(q/\phi(q))^C$ for $C=C(k)$ and $\Phi(q,\l)\ll_{k,\l} 1$ for $\l\geq 3$. Further, for $\alpha \in \mathfrak{m}(\gamma,Q)$ we have 
\begin{equation} \label{SminII} |S(\alpha)| \leq  \delta T/16.
\end{equation} The proof of the estimates in Theorem \ref{standalonethm} and the subsequent deduction of (\ref{SmajII}) and (\ref{SminII}) can be found in Section \ref{expest}. From (\ref{SminII}) and Plancherel's Identity, we have \begin{equation*}  \int_{\mathfrak{m}(\gamma,Q)} |\widehat{f_A}(\alpha)|^2|S(\alpha)|d\alpha \leq \delta^2NT/16, \end{equation*} which together with (\ref{mass}) yields \begin{equation}\label{majmass}  \int_{\mathfrak{M}(\gamma,Q)}|\widehat{f_A}(\alpha)|^2|S(\alpha)|d\alpha \geq \delta^2 NT/16. \end{equation} 
From (\ref{SmajII}) and (\ref{majmass}) , we have 
\begin{equation} \label{majmassII} \sum_{q=1}^Q  (k-1)^{\l\omega(q)}\Phi(q)q^{-\ell/2} \int_{\mathbf{M}_q(\gamma)}|\widehat{f_A}(\alpha)|^2 {d}\alpha \gg_{h} \delta^2N.
\end{equation}
For $\ell=2$, the function $b(q)=(k-1)^{2\omega(q)}(q/\phi(q))^C$ satisfies $b(qr)\geq b(r)$, and we make use of the following proposition, which is based on a trick that originated in \cite{Ruz}.
\begin{proposition}[Proposition 5.6, \cite{ricemax}] \label{rstrick} For any $\gamma,Q>0$ satisfying $2\gamma Q^2<1$, and for any function $b: \N \to [0,\infty)$ satisfying $b(qr)\geq b(r)$ for all $q,r\in \N$, we have $$\max_{q\leq Q} \int_{\mathbf{M}'_q(\gamma)}|\widehat{f_A}(\alpha)|^2 {d}\alpha \geq Q \Big(2\sum_{q=1}^Q b(q)\Big)^{-1} \sum_{r=1}^Q \frac{b(r)}{r}\int_{\mathbf{M}_r(\gamma)}|\widehat{f_A}(\alpha)|^2 {d}\alpha. $$
\end{proposition}

\noindent Because $b$ is a multiplicative function, $b(p^v)=(k-1)^2(1+1/(p-1))^C\ll_k 1$ for all prime powers $p^v$, and $$\sum_{q=1}^Q \frac{b(q)}{q}\leq \prod_{p\leq Q} \left(1+\frac{b(p)}{p}+\frac{b(p)}{p^2}+\cdots\right)=\prod_{p\leq Q} \left(1+\frac{(k-1)^2}{p}+O_k(1/p^2)\right)\ll_k \log^{(k-1)^2}Q, $$ it follows from Theorem 01 of \cite{HallTen} that 
\begin{equation*} \sum_{q=1}^Q b(q) \ll_k Q\log^{(k-1)^2-1} Q, 
\end{equation*}
and the lemma for $\ell=2$ follows from (\ref{majmassII}) and Proposition \ref{rstrick}. For $\l \geq 3$, since $(k-1)^{\l\omega(q)}\ll_{k,\l,\epsilon} q^{\epsilon}$ for any $\epsilon>0$, the sum $\sum_{q=1}^{\infty}  (k-1)^{\l\omega(q)}q^{-\ell/2}$ is convergent, and hence (\ref{majmassII}) immediately yields $$\max_{q\leq Q} \int_{\mathbf{M}_q(\gamma)}|\widehat{f_A}(\alpha)|^2 {d}\alpha \gg_h \delta^2 N.$$ Since $\mathbf{M}_q(\gamma)\subseteq \mathbf{M}'_q(\gamma), $ this establishes the lemma for $\ell \geq 3$. \qed

\section{Criteria for strongly Deligne polynomials} \label{AG1}
In this section, we prove Proposition~\ref{thm:main} and a stronger version of Proposition~\ref{prop:main}. We begin, though, by collecting a few facts from algebraic geometry that will be useful in subsequent sections. Throughout this section, for a variety $V$, we let $V^\sing$ denote the singular locus of $V$, and we let $V^\ns=V\setminus V^\sing$.

\subsection{Results from algebraic geometry}

We first state a classical version of B\'ezout's Theorem; see \cite[Example 8.4.6]{Fulton}.

\begin{lemma}[B\'ezout's Theorem]\label{lem:bezout}
Let $V_1,\ldots,V_k$ be subvarieties of $\bbP^\ell$. Then $\deg\bigcap_{i=1}^k V_i \le \prod_{i=1}^k \deg V_i.$
In particular, if the intersection is finite, then $\left| \bigcap_{i=1}^k V_i \right| \le \prod_{i=1}^k \deg V_i.$
\end{lemma}

We now record estimates due to Lang and Weil \cite{LangWeil} on the number of points on varieties over finite fields. The following is a well known consequence of Theorem 1 of \cite{LangWeil} (see, for example, Theorem 5.1 of \cite{PoonenSlavov}), but we give the short proof for completeness.

\begin{lemma} \label{langweillem}
Let $k$, $\ell$, $m$, and $r$ be positive integers, and let $q$ be a prime power. Let $V$ be a (reduced) closed subvariety of $\bbP^\ell$, defined over $\F_q$ (the field with $q$ elements), of degree $k$ and dimension $r$. Let $m \ge 1$ be the number of geometrically irreducible components of $V$ which are defined over $\F_q$. Then
	\begin{equation}\label{eq:LangWeil}
		|V(\F_q)|,\ |V^\ns(\F_q)| = mq^r + O_{k,\ell,r}(q^{r-1/2}).
	\end{equation}
Moreover, the same is true if we replace $V$ with a closed subvariety $W \subseteq \bbA^\ell$.
\end{lemma}

\begin{proof}
The proof is by induction on $r$, noting that the case $r = 0$ is elementary, and amounts to considering the following observations.
	\begin{enumerate}[1.]
	\item If $P \in Z(\F_q)$ for a component $Z \subset V$ not defined over $\F_q$, then $P = P^\sigma \in Z^\sigma \ne Z$ for nontrivial $\sigma \in \Gal(\overline{\F_q}/\F_q)$, hence $P \in Z \cap Z^\sigma$, which has dimension strictly less than $r$. Thus, the number of points on components not defined over $\F_q$ is absorbed into the error term.
	
	\item Each component of $V$ defined over $\F_q$ has $q^r + O_{k,\ell,r}(q^{r-1/2})$ by Theorem 1 of \cite{LangWeil}. Summing the number of points on each component is an overcount, but the surplus is due to points on pairwise intersections of components, which again is absorbed into the error term. (Note that $m \le k$, so even after multiplying the error by $m$, the implied constant still depends only on $k$, $\ell$, and $r$.) Thus $|V(\F_q)|$ has the claimed magnitude.
	
	\item We have $V^\ns := V \setminus V^\sing$; since $V^\sing$ has dimension at most $r - 1$ and degree controlled by $k$, $r$, and $\ell$ (by B\'ezout's Theorem), the size of $V^\sing(\F_q)$ is included in the error term. Thus, $|V^\ns(\F_q)|$ also has the desired magnitude.
	
	\item Finally, if we let $V$ be the projective closure of $W$, then $W = V \setminus (V \cap H)$, where $H$ is the hyperplane at infinity. Since $V \cap H$ has lower dimension and degree $k$, we are once again removing a set whose cardinality is subsumed by the error term, so $W(\F_q)$ (and, similarly, $W^\ns(\F_q)$) has the appropriate cardinality. \qedhere
	\end{enumerate}
\end{proof}

\subsection{A key equivalence} The following equivalence observation yields a strengthening of Proposition \ref{prop:main} as a corollary, and is also instrumental in subsequent proofs.

\pagebreak

\begin{lemma}\label{lem:equiv}
Let $V$ be a variety (reduced, but not necessarily irreducible) of dimension $d \ge 1$ defined over $\Z$. For a sufficiently large (with respect to $V$) prime $p$, the following are equivalent:
	\begin{enumerate}[(a)]
	\item $V^\ns(\F_p) \ne \emptyset$.
	\item $V^\ns(\Z_p) \ne \emptyset$.
	\item At least one of the geometric components of $V$ is defined over $\Z_p$.
	\end{enumerate}
\end{lemma} 
 
\begin{proof}\mbox{}

\noindent((a) $\implies$ (b)) Suppose $V^\ns(\F_p) \ne \emptyset$, and let $Q \in V^\ns(\F_p)$. By Hensel's lemma, there exists $P \in V(\Z_p)$ such that $\widetilde{P} = Q$. Since $\widetilde{P}$ is nonsingular, so must be $P$.
 
\noindent ((b) $\implies$ (c))
Let $P \in V^\ns(\Z_p)$, and let $Z$ be a geometric component of $V$ containing $P$. As in part 1 of the proof of Lemma~\ref{langweillem}, if $Z$ were not defined over $\F_q$, then $P$ would lie in the intersection of two components, hence would be a singular point on $V$, contradicting our assumption on $P$.


\noindent ((c) $\implies$ (a)) 
Let $Z_1,\ldots,Z_m$ be the irreducible components of $V$. By Lemma \ref{langweillem}, for each $1 \le i \le m$ there exists a bound $B_i$ such that for all $p \ge B_i$ with $Z_i$ defined over $\Z_p$, $Z_i^\ns(\F_p)$ contains a point that does not lie on $Z_j$ for any $j \ne i$. Letting $B = \max\{B_1,\ldots,B_m\}$, we have that for $p \ge B$, the existence of $Z_i$ defined over $\Z_p$ implies the existence of $Q \in Z_i^\ns(\F_p) \setminus \bigcup_{j\ne i} Z_j(\F_p)$. Since $Q$ is nonsingular on $Z_i$ and is not a point of intersection with any other component $Z_j$, we have $Q\in V^\ns(\F_p)$. 
\end{proof}  

As previously noted, if $h\in \Z[x_1,\dots,x_{\l}]$ is Deligne, then $h=0$ defines a reduced variety. Further, a nonsingular point over $\Z_p$ on this variety corresponds precisely to a root $\bsz_p\in \Z_p^{\l}$ of $h$ satisfying $m_p=1$, hence Lemma \ref{lem:equiv} establishes the following sufficient condition for smooth intersectivity. Here we let $\ZZbar$ denote the ring of algebraic integers.
 
\begin{corollary}\label{prop:main_stronger}
Suppose $\ell \ge 2$ and $h \in \Z[x_1,\ldots,x_{\ell}]$ is Deligne and intersective, and let $h=g_1\cdots g_n$ be an irreducible factorization of $h$ in $\ZZbar[x_1,\dots,x_{\l}]$. If, for all but finitely many $p\in \P$, $g_i$ has coefficients in $\Z_p$ for some $1\leq i \leq n$, then $h$ is smoothly intersective, hence strongly Deligne.
\end{corollary}
 
Note that Proposition~\ref{prop:main} is an immediate consequence of Corollary~\ref{prop:main_stronger}, since the hypotheses of the proposition imply that one of the factors over $\ZZbar$ is defined over $\Z$, hence over $\Z_p$ for all $p$. We now complete this section by using Lemma \ref{lem:equiv} to prove Proposition \ref{thm:main}.

\begin{proof}[Proof of Proposition~\ref{thm:main}]
Let $\l\geq 2$, and suppose $h\in \Z[x_1,\dots,x_{\l}]$ is Deligne and intersective with $\deg(h)=k\geq 2$. Let $\{\bsz_p\}_{p\in \P}$ be a choice of $p$-adic integer roots of $h$ satisfying $m_p\in \{1,k\}$ for all but finitely many $p$. 
Let $X$ denote the finite set of primes such that \\[-20pt]

	\begin{itemize}
	\item $m_p\notin \{1,k\}$, or \\[-18pt]
	\item $p\mid k$, or \\[-18pt]
	\item $h^k$ is not smooth modulo $p$, or \\[-18pt]
	\item the equivalence in Lemma \ref{lem:equiv} fails.
	\end{itemize}
\noindent We note that the first item is assumed to be finite, the second item is clearly finite, the fourth item is proven to be finite in Lemma \ref{lem:equiv}, and the third item is finite because $h$ is Deligne (see Definition~\ref{defn:smoothDeligne}). 

\noindent Fix $d\in \N$ and $p\notin X$. If $p\nmid d$ or $m_p=k$, then $p\nmid \frac{d^k}{\lambda(d)}$, so $h_d^k=\frac{d^k}{\lambda(d)}h^k$ is a nonzero scalar multiple of $h^k$, hence remains smooth modulo $p$. Therefore, $h_d$ is Deligne modulo $p$.

\noindent The remaining case is $p\mid d$ and $m_p=1$. In this case, since $h^i_d$ has a factor of $\frac{d^i}{\lambda(d)}$, the definition of $\lambda$ assures that the polynomial $h^i_d$ identically vanishes modulo $p$ for all $i>1$. Since nonzero homogeneous linear polynomials are automatically smooth, we need only argue that $$h_d^1(\bsx)=\frac{d}{\lambda(d)}\sum_{i=1}^{\l}\frac{\partial h}{\partial x_i} (\bsr_d)x^i$$ does not identically vanish modulo $p$. We know that $p\nmid \frac{d}{\lambda(d)}$ by definition of $\lambda$. Further, the fact that $h$ is Deligne ensures that $h=0$ defines a reduced variety, so by Lemma \ref{lem:equiv}, we can choose $\bsz_p$ to reduce to a nonsingular point over $\F_p$. Since $\bsr_d\equiv \bsz_p \ (\text{mod }p)$, we have that $\frac{\partial h}{\partial x_i} (\bsr_d)\equiv \frac{\partial h}{\partial x_i} (\bsz_p)\not\equiv 0 \ (\text{mod }p)$ for some $1\leq i \leq \l$, as required. Therefore, $h_d$ is Deligne modulo $p$ for all $p\notin X$, hence $h$ is strongly Deligne. 
\end{proof}

\section{Dimension lowering argument}\label{dimlowsec}

In this section, we generalize the phenomenon exemplified at the beginning of Section \ref{sec:singular}, establishing Theorems \ref{dimlowrootthm} and \ref{dimlowthm} by reducing to the case covered in Theorem \ref{more}. In the integer root setting, this reduction is very direct, as Theorem \ref{dimlowrootthm} follows immediately from Theorem \ref{more} and the following proposition.

\begin{proposition}\label{prop:dim_lower_0}
Suppose $\l \ge 2$ and $h \in \Z[x_1,\ldots,x_{\l}]$ with $h(\boldsymbol{0}) = 0$. Let $r$ be the minimum rank of the highest and lowest degree homogeneous parts of $h$. If $r\geq 2$, then there exists a strongly Deligne polynomial $g \in \Z[x_1,\ldots,x_r]$  such that $g(\Z^r) \subseteq h(\Z^{\l})$.
\end{proposition}

Before delving into the proof of this proposition, we state a version of Bertini's theorem that will allow us to eliminate the singularity in the top-degree parts of our polynomials, one dimension at a time. 
Throughout this section we let $(\bbP^n)^*$ denote the dual space of $\bbP^n$, that is, the space of hyperplanes in $\bbP^n$. Note that $(\bbP^n)^*$ is isomorphic to $\bbP^n$, with the hyperplane $\{a_0x_0 + \cdots + a_nx_n = 0\} \in (\bbP^n)^*$ corresponding to the point $(a_0 : \cdots : a_n) \in \bbP^n$. A {\it linear system of hyperplanes in $\bbP^n$} is a linear subspace of $(\bbP^n)^*$.

\pagebreak

\begin{theorem}[Bertini's Theorem]\label{thm:bertini}
Let $V$ be a (quasi-projective) subvariety of $\bbP^n$ with irreducible components $V_1,\ldots,V_m$ of equal dimension $d \ge 1$, and let $\calL \subseteq (\bbP^n)^*$ be a linear system. After a change of coordinates if necessary, we may assume that there exists $k \in \{0,\ldots,n\}$ such that $\calL$ is the space of all hyperplanes of the form $\{a_kx_k + \cdots + a_nx_n = 0\}$. Assume that the coordinates $x_k, \ldots, x_n$ do not simultaneously vanish at any point on $V$ (i.e., the linear system $\calL$ has no base-points in $V$), so that
	\begin{align*}
	\Phi_\calL : V &\longrightarrow \bbP^{n-k}\\
	(z_0 : \cdots : z_n) &\longmapsto (z_k : \cdots : z_n)
	\end{align*}
defines a morphism.
Then there exists a nonempty open subset $U \subseteq \calL$ such that for all hyperplanes $H \in U$,
	\begin{enumerate}[(a)]
	\item $V^\ns \cap H$ is nonsingular, and
	\item $\dim \left(V^\sing \cap H\right) < \dim V^\sing$ (if $V^\sing \neq \emptyset$).
	\end{enumerate}

\noindent Moreover, if $\dim \Phi_\calL(V) \ge 2$, then $U$ may be chosen so that for all $H \in U$ we have
	\begin{enumerate}[(a)]
	\setcounter{enumi}{2}
	\item for all $1 \le i \le m$, the intersection $V_i \cap H$ is either empty or geometrically irreducible.
	\end{enumerate}
\end{theorem}

\begin{rem}
Theorem~\ref{thm:bertini} is stated somewhat more generally than we need, so we specify the two situations for which we will actually need the result:
	\begin{enumerate}[1.]
	\item Let $V$ be a closed hypersurface in $\bbP^n$ and let $\calL = (\bbP^n)^*$. Then $\Phi_\calL$ is just the inclusion map of $V$ into $\bbP^n$, and the hypotheses of Theorem~\ref{thm:bertini} are satisfied. Moreover, since each component $V_i$ is a closed subvariety of $\bbP^n$ of positive dimension, each intersection $V_i \cap H$ is nonempty; thus, if $d = \dim V = \dim \Phi_\calL(V) \ge 2$, then $V_i \cap H$ is irreducible for all $1 \le i \le m$ and all $H \in U$.
	\item Identify $\bbA^n$ with the Zariski open subset $\{x_0 \ne 0\} \subset \bbP^n$. Let $V$ be a closed hypersurface in $\bbA^n$ {\it not containing the origin $\boldsymbol{0} = (0,\ldots,0)$}, and let $\calL$ be the space of all hypersurfaces of the form $\{a_1x_1 + \cdots + a_nx_n = 0\}$. Then the conditions of Theorem~\ref{thm:bertini} are satisfied once again. A fiber of $\Phi_\calL$ is precisely the intersection of $V$ with a line in $\bbA^n$ passing through $\boldsymbol{0}$. Since $V$ is closed in $\bbA^n$ and does not contain $\boldsymbol{0}$, $V$ cannot contain a line through $\boldsymbol{0}$, hence each such intersection is finite. In particular, this means the map $\Phi_\calL$ has finite fibers, so $\dim \Phi_\calL(V) = \dim V = d$. Moreover, the failure of a hyperplane $H \in \calL$ to intersect every $V_i$ is a proper Zariski closed condition.
	Therefore, removing such hyperplanes from $U$ if necessary, we again have that $V_i \cap H$ is nonempty for all $1 \le i \le n$ and $H \in U$, thus $V_i \cap H$ is irreducible for all $1 \le i \le n$ and $H \in U$ as long as $d \ge 2$.
	\end{enumerate}
\end{rem}

%

\begin{proof}[Proof of Theorem~\ref{thm:bertini}]
Consider the set $\calX$ of hyperplanes $H \in \calL$ satisfying the following conditions:
	\begin{enumerate}[(a$'$)]
	\item $V_i^\ns \cap H$ is nonsingular for all $1 \le i \le m$;
	\item $H$ does not contain any components of $V_i^\sing$ nor $(V_i \cap V_j)$ for $1 \le i,j \le m$ with $i \ne j$; and
	\item for all $1 \le i \le m$, the intersection $V_i \cap H$ is either empty or geometrically irreducible (if $d \ge 2)$.
	\end{enumerate}

\noindent We begin by showing that if $H \in \calX$, then $H$ satisfies properties (a), (b), and (c). Indeed, condition (c$'$) is exactly condition (c), so we need only show that $H$ also satisfies (a) and (b).

\noindent For (a), note that a point $P \in V$ is nonsingular if and only if $P$ is a nonsingular point on $V_i$ for some $1 \le i \le m$ and $P \notin V_j$ for all $j \ne i$. Thus $V^\ns$ is a {\it disjoint} union $V^\ns = \bigsqcup_{i=1}^m W_i$, where each $W_i$ is a subset of $V_i^\ns$. Then (a) follows from (a$'$) since $V^\ns \cap H = \bigsqcup_{i=1}^m (W_i \cap H)$ and each $W_i \cap H \subseteq V_i^\ns \cap H$ is nonsingular. Finally, (b$'$) implies that $H$ intersects each component of $V^\sing$ properly (assuming $V^\sing \ne\emptyset$), so (b) follows.

\noindent It remains to show that $\calX$ contains a Zariski open subset of $\calL$.
By the standard form of Bertini's Theorem (see Corollaire 6.11 of \cite{Jouanolou}, or Corollary 10.9 and Remark 10.9.1 of \cite{Hartshorne}), since $\calL$ has no base-points in $V$, the set of hyperplanes $H \in \calL$ satisfying (a$'$) and (c$'$) contains a nonempty open subset of $\calL$. Moreover, $H$ containing any of a finite collection of nonempty subvarieties of $\bbP^n$ is a proper closed condition on $H$, so condition (b$'$) is a nonempty open condition; therefore, $\calX$ contains a nonempty open subset of $\calL$.
\end{proof}
Armed with Theorem~\ref{thm:bertini}, the proof of Proposition \ref{prop:dim_lower_0} is pleasingly straightforward.
\begin{proof}[Proof of Proposition~\ref{prop:dim_lower_0}]
Suppose that  $\l \ge 2$ and $h \in \Z[x_1,\ldots,x_{\l}]$ with $h(\boldsymbol{0}) = 0$. Let $k$ and $j$ denote the highest and lowest degrees, respectively, of the terms appearing in $h$, and let $r$ denote the minimum rank of $h^k$ and $h^j$. Let $\Vhat_k,\Vhat_j\subseteq \bbP^{\l-1}$ denote the varieties defined by $h^k=0$ and $h^j=0$, respectively. By 
Theorem~\ref{thm:bertini} (see also case 1 of the remark that follows)
applied to the linear system $\calL = (\bbP^{\l - 1})^*$ and the varieties $\Vhat_k$ and $\Vhat_j$, respectively, the set of hyperplanes $H$ in $\bbP^{\l - 1}$ satisfying \\[-20pt]
	\begin{itemize}
	\item $H \cap \Vhat_k^\ns$ and $H \cap \Vhat_j^\ns$ are nonsingular, and \\[-18pt]
	\item $\dim (H \cap \Vhat_k^\sing) < \dim \Vhat_k^\sing$, if $V^\sing \neq \emptyset$, and $\dim (H \cap \Vhat_j^\sing) < \dim \Vhat_j^\sing$, if $\Vhat_j^\sing\neq \emptyset$, 
	\end{itemize}
contains a nonempty open subset $U \subseteq (\bbP^{\l-1})^*$. Thus, we can choose $H \in U$ defined by the vanishing of $l(x_1,\ldots,x_{\l}) = a_1x_1 + \cdots + a_{\l-1}x_{\l-1} - x_{\l}$ with $a_1,\ldots,a_{\l-1} \in \Z$. Here, we're using the fact that the set of integer points is Zariski dense in the affine space $\bbA^{\l-1} \subset \bbP^{\l-1} \cong (\bbP^{\l-1})^*$.

\noindent Let $\mu(x_1,\ldots,x_{\l-1}) := a_1x_1 + \cdots + a_{n-1}x_{\l-1}$, and set
	\[
		g_1(x_1,\ldots,x_{\l-1}) := h(x_1,\ldots,x_{\l-1},\mu(x_1,\ldots,x_{\l-1})).
	\]
Note that, by construction, $g_1(\Z^{\l-1}) \subseteq h(\Z^\l)$, $g_1(\bszero)=0$, and the highest and lowest degrees of the nonzero terms of $g_1$ are still $k$ and $j$, respectively.

\noindent Now, the subvariety $\What_k$ (resp., $\What_j$) of $\bbP^{\l-2}$ defined by $g_1^k = 0$ (resp., $g_1^j = 0$) is isomorphic to $H \cap \Vhat_k$ (resp., $H \cap \Vhat_j$). In particular, the minimum rank of $g_1^k$ and $g_1^j$ can only drop below $r$ if both singular loci were originally empty, which would imply $r=\l$. Thus, repeating this process $(\l - r)$ times yields a sequence of polynomials $\left(g_i(x_1,\ldots,x_{\l-i})\right)_{i=0}^{\l-r}$, with $g_0 := h$, satisfying \\[-20pt]
	\begin{itemize}
	\item $g_i(\Z^{\l-i}) \subseteq g_{i-1}(\Z^{\l-i+1})$ for all $1 \le i \le \l - r$,  \\[-18pt]
	\item $g_i(\bszero)=0$ for all $0 \le i \le \l - r$, \\[-18pt]
	\item the highest and lowest degrees of the nonzero terms of each $g_i$ are $k$ and $j$, respectively, and \\[-18pt]
	\item the rank of each $g_i^k$ (resp., $g_i^j$) is at least $r$.
	\end{itemize}
Finally, let $g := g_{\l-r}\in \Z[x_1,\dots,x_r]$, so the rank for each of $g^k$ and $g^j$ is $r$. In other words, $g^k$ and $g^j$ are smooth, and thus by Proposition \ref{prop:integer_root}, $g$ is strongly Deligne.
\end{proof}

\begin{rem} 
The conclusion of Proposition~\ref{prop:dim_lower_0} technically holds for $r = 1$ as well, since nonconstant univariate polynomials are necessarily Deligne; however, this case is not useful for our purposes.
\end{rem}

\subsection{Proof of Theorem \ref{dimlowthm}} We now proceed with the more elaborate of our two dimension-lowering arguments, in which we cannot exploit the existence of an integer root. Throughout this section, we fix $\l\geq 2$ and a polynomial $h\in \Z[x_1,\dots,x_{\l}]$ satisfying all hypotheses of Theorem \ref{dimlowthm}, and we recall that $k=\deg(h)$ and $r$ denotes the rank of $h^k$. Note that the hypotheses of Theorem~\ref{dimlowthm} imply that $r \ge 2$. We assume without loss of generality that $h(\bszero)\neq 0$, which is permissible because $h(\Z^{\l})$ is invariant under input translation. We let $V\subseteq \bbA^{\l}$ and $\Vhat \subseteq \bbP^{\l-1}$ denote the varieties defined by $h=0$ and $h^k=0$, respectively. The following crucial lemma says that we can eliminate the singularity in the top-degree part of $h$, one dimension at a time, while maintaining the existence of nonsingular $\F_p$-points.

\begin{lemma}\label{lem:lin_poly}
Suppose $\ell \ge 3$. Then there exists a homogeneous linear polynomial $l \in \Z[x_1,\ldots,x_{\l}]$, monic in $x_{\l}$, for which the following holds: Letting $\Lhat$ and $L$ denote the hyperplanes in $\bbP^{\l - 1}$ and $\bbA^\l$, respectively, defined by $l = 0$, we have
	\begin{enumerate}[(i)]
	\item $\dim (\Vhat \cap \Lhat)^\sing < \dim \Vhat^\sing$, if $\Vhat^\sing\neq \emptyset$; and
	\item For sufficiently large $p$, $V^\ns(\F_p)\neq \emptyset$ implies $(V \cap L)^\ns(\F_p) \ne\emptyset$.
	\end{enumerate}
\end{lemma}

\begin{proof}
Let $\calLhat$ and $\calL$ denote the linear systems of {\it hyperplanes in $\bbP^{\ell-1}$} and {\it hyperplanes in $\bbA^{\ell}$ passing through $\bszero$}, respectively. We identify each of $\calLhat$ and $\calL$ with $\bbP^{\ell-1}$, with the point $\bsa = (a_1 : \cdots : a_\ell) \in \bbP^{\ell-1}$ corresponding to the hyperplanes $\{a_1x_1 + \cdots + a_\ell x_\ell = 0\}$ in $\bbP^{\ell - 1}$ and $\bbA^\ell$, respectively.

\noindent 
The hypotheses of Theorem~\ref{thm:bertini} are satisfied by $\Vhat$ and $\calLhat$ (resp., $V$ and $\calL$), as explained in case 1 (resp., case 2) of the remark immediately following the theorem. Thus, there is a nonempty open set $U \subseteq \bbP^{\ell-1}$ such that for all $\bsa = (a_1 : \cdots : a_\ell) \in U$, the hyperplanes $\Lhat_\bsa \subset \bbP^{\ell - 1}$ and $L_\bsa \subset \bbA^\ell$ defined by $l_\bsa := a_1x_1 + \cdots + a_\ell x_\ell = 0$ satisfy the conclusion of Theorem~\ref{thm:bertini} (intersected with $\Vhat \subset \bbP^{\ell - 1}$ and $V \subset \bbA^\ell$, respectively). Similar to the proof of Proposition~\ref{prop:dim_lower_0}, we may choose $\bsa \in U$ of the form $\bsa = (a_1 : \cdots : a_{\ell - 1} : 1)$ with $a_1,\ldots,a_{\ell-1} \in \Z$. Set $l := l_\bsa$ for such a choice of $\bsa \in U$, hence also $\Lhat = \Lhat_\bsa$ and $L = L_\bsa$. By construction, we immediately have that $l \in \Z[x_1,\ldots,x_\ell]$, $l$ is monic in $x_\ell$, and property (i) holds, so it remains only to show that (ii) holds.

\noindent Let $V_1,\ldots,V_m$ be the geometrically irreducible components of $V$. Since $\dim V = \l - 1 \ge 2$, our choice of $\bsa$ guarantees that the geometrically irreducible components of $V \cap L$ are $V_i \cap L$ with $1 \le i \le m$. (We are again using Theorem~\ref{thm:bertini} and case 2 of the remark that follows.) By Lemma~\ref{lem:equiv}, if $p$ is sufficiently large, then $V^\ns(\F_p) \ne\emptyset$ implies that $V_i$ is defined over $\Z_p$ for some $1 \le i \le m$. Since $L$ is defined over $\Z$, hence over $\Z_p$, the intersection $V_i \cap L$ is also defined over $\Z_p$. Appealing to Lemma~\ref{lem:equiv} once more implies that $(V \cap L)^\ns(\F_p)$ is nonempty.
\end{proof}

The hyperplane produced by Lemma \ref{lem:lin_poly} quickly yields a suitable polynomial with one fewer variable.

\begin{corollary}\label{cor:one_step}
Suppose $\l \ge 3$. Then there exists $g' \in \Z[x_1,\ldots,x_{\l-1}]$ with $\deg(g')=k$  such that
	\begin{enumerate}[(i)]
	\item $g'(\bszero) \ne 0$;
	\item $g'(\Z^{\l-1}) \subseteq h(\Z^\l)$; 
	\item $\dim (\What')^\sing < \dim \Vhat^\sing$, if $\Vhat^\sing \neq \emptyset$; and
	\item for sufficiently large $p$, $V^\ns(\F_p)\neq \emptyset$ implies $(W')^\ns(\F_p) \ne\emptyset$;
	\end{enumerate}
where $\What' \subset \bbP^{\l-2}$ and $W' \subset \bbA^{\l-1}$ are the varieties defined by $(g')^k = 0$ and $g' = 0$, respectively.
\end{corollary}

\begin{proof}
Let $L = L_\bsa$ be as in Lemma~\ref{lem:lin_poly}, and write $l = l_\bsa = a_1x_1 + \cdots + a_{\l-1}x_{\l-1} + x_{\l}$. To ease notation, we also set $\mu = \mu_\bsa := -(a_1x_1 + \cdots + a_{\l-1}x_{\l-1})$. Now, define
	\[
		g'(x_1,\ldots,x_{\l-1}) := h(x_1,\ldots,x_{\l-1},\mu(x_1,\ldots,x_{\l-1})).
	\]
Clearly $g'(\Z^{\l-1}) \subseteq h(\Z^{\l})$ and, since $\mu$ is homogeneous, $g'(\bszero) = h(\bszero) \ne 0$. 
Finally, since $V \cap L \cong W'$ and $\Vhat \cap \Lhat \cong \What'$, properties (iii) and (iv) follow immediately from Lemma~\ref{lem:lin_poly}.
\end{proof}

Recall our assumption that the rank satisfies $r \ge 2$. Repeated application of Corollary~\ref{cor:one_step} yields the following:

\begin{corollary}\label{cor:all_steps}
There exists $g \in \Z[x_1,\ldots,x_{r}]$ with $\deg(g)=k$ such that
	\begin{enumerate}[(i)]
	\item $g(\Z^r) \subseteq h(\Z^{\l})$;
	\item $g$ is Deligne; and
	\item for sufficiently large $p$, $V^\ns(\F_p)\neq \emptyset$ implies $W^\ns(\F_p) \ne\emptyset$;
	\end{enumerate}
where $W \subseteq \bbA^r$ is the variety defined by $g = 0$.
\end{corollary}

\begin{proof}
When $\l \ge 3$, this follows immediately by applying Corollary~\ref{cor:one_step} recursively $(\l - r)$ times. The fact that $r \ge 2$ ensures that at each step we are applying Corollary~\ref{cor:one_step} to a polynomial in at least $3$ variables.

\noindent When $\l = 2$, the statement is trivial, since $r = 2$ implies that $h$ is already Deligne, so we can take $g = h$.
\end{proof}

\begin{rem}
Using the construction from the proof of Corollary~\ref{cor:one_step}, the polynomial $g$ of Corollary~\ref{cor:all_steps} may be written in the form
	\[
		g(x_1,\ldots,x_r) = h(x_1,\ldots,x_r,\mu_{r+1}(x_1,\ldots,x_r), \ldots, \mu_{\l}(x_1,\ldots,x_r)),
	\]
where each $\mu_j$ is a homogeneous linear polynomial. We will use this precise form in our proof of Theorem~\ref{dimlowthm}, which we are now ready to begin.
\end{rem}

\begin{proof}[Proof of Theorem \ref{dimlowthm}]

\noindent Let $g \in \Z[x_1,\ldots,x_r]$ be as in Corollary~\ref{cor:all_steps}, and let $W \subseteq \bbA^r$ be the variety defined by $g = 0$. Throughout this proof we use the notation $\tilde{\bsx}=(x_1,\dots,x_r)$ and $\bsx=(x_1,\dots,x_{\l})$.

\noindent As mentioned in the remark above, $g$ may be given by $g(\tilde{\bsx})=h(M\tilde{\bsx})$, where $$M(x_1,\dots,x_r)=(x_1,\dots,x_r,\mu_{r+1}(x_1,\dots,x_r), \dots, \mu_{\l}(x_1,\dots,x_r)) $$ for linear forms $\mu_{r+1},\dots,\mu_{\l}$. Note that $g$ and the linear forms have been constructed once and for all from $h$, so any quantities depending on them implicitly depend only on $h$. 

\noindent Let $X = X(h)$ be the set of primes $p$ for which 
	\begin{itemize} 
	\item $p\mid k$;
	\item $g^k$ is not smooth modulo $p$; or
	\item $W^\ns(\F_p) = \emptyset$ and $m_p\neq k$ for all $\bsz_p\in V(\Z_p)$. 
	\end{itemize}
 The first item clearly defines a finite set, the second item defines a finite set because $g$ is Deligne (see Definition~\ref{defn:smoothDeligne}). If $r\geq 3$, then the third item defines a finite set by Lemma \ref{lem:equiv} and the fact that Deligne polynomials in $r\geq 3$ variables are geometrically irreducible, as seen in the proof of Corollary \ref{3var}. If $r=2$, then item (iii) of Corollary \ref{cor:all_steps}, Lemma \ref{lem:equiv}, and the hypotheses of Theorem \ref{dimlowthm} ensure that the third item defines a finite set. Thus, $X$ is finite.
 
\noindent In order to construct auxiliary polynomials $h_d$ for $d \in N$, we first choose $\Z_p$-roots of $h$ as follows: If $p \in X$, then choose a point $\bsz_p \in V(\Z_p)$ arbitrarily; such points exist because $h$ is intersective. For $p \notin X$ with $W^\ns(\F_p)\neq \emptyset$, choose $\tilde{\bsz}_p \in W(\Z_p)$ to be a Hensel lift of a nonsingular point on $W(\F_p)$, then set $\bsz_p=M\tilde{\bsz}_p\in V(\Z_p).$ Finally, for all remaining $p\notin X$, fix $\bsz_p \in V(\Z_p)$ with $m_p=k$. 

\noindent For each prime $p$, by definition of multiplicity, we have a decomposition of the form \begin{equation}\label{hdecomp} h(\bsx+\bsz_p)=\sum_{m_p\leq |\bsi| \leq k} b_{\bsi}\bsx^{\bsi} \end{equation} for $b_{\bsi}\in \Z_p$. However, the substitution $\bsx=M\tilde{\bsx}$ could cause some homogeneous parts to identically vanish, so we define $\tilde{m}_p$ to be the multiplicity of $\bszero$ as a root of $h(M\tilde{\bsx}+\bsz_p)$, so in particular \begin{equation}h(M\tilde{\bsx}+\bsz_p)=\sum_{m_p\leq |\bsi| \leq k} b_{\bsi}(M\tilde{\bsx})^{\bsi}=\sum_{\tilde{m}_p\leq |\bsi|\leq k} a_{\bsi} \tilde{\bsx}^{\bsi},
\end{equation} where $a_{\bsi}\neq 0$ for some $\bsi$ with $|\bsi|=\tilde{m}_p$. We quickly note that $\tilde{m}_p=m_p$ for all $p\notin X$. If $p\notin X$ with $m_p=k$, the degree-$k$ part of $h(M\tilde{\bsx}+\bsz_p)$ is the same as the degree-$k$ part of $g$. If $p\notin X$ and $\bsz_p=M\tilde{\bsz}_p$ as above, then $h(M\tilde{\bsx}+\bsz_p)$ is precisely $g(\tilde{\bsx}+\tilde{\bsz}_p)$, and in particular the linear part does not vanish modulo $p$.

\noindent To account for this possible increase in multiplicity for primes $p\in X$, we define a completely multiplicative function $\tilde{\lambda}(d)$ by setting $\lambda(p)=p^{\tilde{m}_p}$ for all primes $p$. We define $\{\bsr_d\}_{d\in \N}$ from $\{\bsz_p\}_{p\in\P}$ as usual from the Chinese remainder theorem, then define the slightly modified auxiliary polynomials $\{\tilde{h}_d\}_{d\in \N}$ by $$\tilde{h}_d(\bsx)=h(\bsr_d+d\bsx)/\tilde{\lambda}(d).$$ We note that $\tilde{h}_d$ can potentially have non-integer coefficients, with denominators divisible by primes in $X$. However, the analog of Proposition \ref{inh}, and the deduction of Lemma \ref{mainit} from Lemma \ref{L2I} and Proposition \ref{inh}, still hold because $d\mid \tilde{\lambda}(d)$ and $\tilde{\lambda}$ is completely multiplicative. 

\noindent We now let $d'=\prod_{p\mid d} p^{(\tilde{m}_p-m_p+1)\text{ord}_p(d)}\leq d^{k},$ and we define $$g_d(\tilde{\bsx})=\tilde{h}_d(\bss_d+M\tilde{\bsx})=h(\bsr_{d'}+dM\tilde{\bsx})/\tilde{\lambda}(d),$$ where $\bss_d$ satisfies $\bsr_{d'}=\bsr_d+d\bss_d$. We will establish the following properties of $g_d$: 

\begin{enumerate}[(i)]
\item $g_d(\Z^r)\subseteq \tilde{h}_d(\Z^{\l})$,

\item $g_d$ has integer coefficients,  

\item $g_d$ is Deligne modulo $p$ for all $p\notin X$,

\item The coefficients of $g_d$ are of size $O_h(d^{k^2})$,

\item $\text{cont}(g_d)\ll_h 1$.

\end{enumerate} 

\noindent Unlike Proposition \ref{prop:dim_lower_0}, these efforts cannot be applied ``externally'' to immediately yield Theorem \ref{dimlowthm} because the family $\{g_d\}_{d\in \N}$ is not necessarily the set of auxiliary polynomials of a single intersective polynomial. However, the enumerated properties of this family make it perfectly suited for us to apply our efforts ``internally'', using the estimates enumerated in Theorem \ref{standalonethm}, as follows: 

\begin{itemize} \item[(1)] Replace all occurrences of $h_d$ in the proof of Theorem \ref{more} with $\tilde{h}_d$. The fact that $\tilde{h}_d$ potentially has non-integer coefficients is not a problem, as the analog of Proposition \ref{inh} still holds, and as explained in the next step.

\item[(2)] When proving Lemma \ref{L2I} (the only piece of the proof of Theorem \ref{more} that requires integer coefficients or a nonsingularity condition), use that $(A-A)\cap g_d(\Z^r) \subseteq (A-A)\cap \tilde{h}_d(\Z^{\l}) \subseteq\{0\},$ then do the remainder of the proof with $h_d$ replaced by $g_d$. For this purpose, properties (ii)-(v) above assure that $g_d$ functions as if it were the auxiliary polynomial of a strongly Deligne polynomial in $r$ variables. In particular, the conclusion of Lemma \ref{L2I} holds with $\theta(k,\l,\delta)$ replaced by $\theta(k,r,\delta)$.

\item[(3)] The remainder of the argument is identical, and Theorem \ref{dimlowthm} follows. 
\end{itemize}

\noindent Our task is now reduced to verifying properties (i)-(v). Properties (i) and (iv) are immediate from the definition of $g_d$ and $\tilde{h}_d$. We next simultaneously establish (ii) and the property 
\begin{equation}\label{ordbound} \text{ord}_p(\text{cont}(g_d))\ll_{h,p} 1  \text{ for all } p\in \P. \end{equation}
When we later establish (iii), it immediately combines with (\ref{ordbound}) to yield (v), because $p\nmid \text{cont}(g_d)$ if $g_d$ is Deligne modulo $p$. We fix $p\in \P$ and set $j=\text{ord}_p(d)$. By (\ref{hdecomp}), we have  
\begin{equation}\label{longeq} g_d(\tilde{\bsx})=\tilde{h}_d(\bss_d+M\tilde{\bsx}) =\frac{1}{\tilde{\lambda}(d)}h(\bsr_{d'}+dM\tilde{\bsx}) =\frac{1}{\tilde{\lambda}(d)}\sum_{m_p\leq |i| \leq k}b_{\bsi}(dM\tilde{\bsx}+\bsr_{d'}-\bsz_p)^{\bsi}. 
\end{equation}
 Since $p^j\mid d$ and $p^{(\tilde{m}_p-m_p+1)j}$ divides all coordinates of $\bsr_{d'}-\bsz_p$, all terms in the summation apart from \begin{equation}\label{vanish} \sum_{m_p\leq|\bsi|\leq \tilde{m}_p-1}b_{\bsi}(dM\tilde{\bsx})^{\bsi} \end{equation} have coefficients divisible by $p^{j\tilde{m}_p}$, and the polynomial (\ref{vanish}) identically vanishes by definition of $\tilde{m}_p$.  Since $\text{ord}_p(\tilde{\lambda}(d))=j\tilde{m}_p$, all coefficients of $g_d$ have nonnegative $p$-adic valuation. Since $p\in \P$ was arbitrary, it follows that $g_d$ has integer coefficients. 

\noindent Further,  we see in (\ref{longeq}) that all degree-$\tilde{m}_p$ terms have a factor of $p^j$ apart from those arising from  $$\frac{d^{\tilde{m}_p}}{\tilde{\lambda}(d)}\sum_{|\bsi|=\tilde{m}_p} b_{\bsi} (M\tilde{\bsx})^{\bsi}= \frac{d^{\tilde{m}_p}}{\tilde{\lambda}(d)}\sum_{|\bsi|=\tilde{m}_p} a_{\bsi} \tilde{\bsx}^{\bsi}, $$
where $a_{\bsi}\neq 0$ for some $\bsi$ with $|\bsi|=\tilde{m}_p$. 

\noindent Since $p\nmid (d^{\tilde{m}_p}/\tilde{\lambda}(d))$, we have that $$\text{ord}_p(\text{cont}(g_d))\leq v:=\min_{|\bsi|=\tilde{m}_p}\text{ord}_p(a_{\bsi}),$$ provided $j>v$. Alternatively, if $j\leq v $, then $\text{ord}_p(\text{cont}(g_d))$ is at most $kv$ plus the minimum $p$-adic valuation of the degree-$k$ coefficients of $g$, which establishes (\ref{ordbound}).

\noindent Our task is now reduced to verifying property (iii), for which we fix $p\notin X$, and proceed similarly to the proof of Proposition \ref{thm:main}. Since $g_d^k$ is precisely $\frac{d^k}{\tilde{\lambda}(d)}g^k$, we know that if $p\nmid d$ or $m_p=k$, then $g_d^k$ modulo $p$ is a nonzero multiple of $g^k$, hence remains smooth. Therefore, $g_d$ is Deligne modulo $p$. 

\noindent The remaining case is when $p\mid d$ and $\bsz_p=M\tilde{z}_p$, where $\tilde{z}_p\in W(\Z_p)$ is a Hensel lift of a nonsingular point of $W(\F_p)$, so in particular the linear part of $g(\tilde{\bsx}+\tilde{z}_p)=h(M\tilde{\bsx}+\bsz_p)$ does not identically vanish modulo $p$. 

\noindent Using (\ref{hdecomp}), letting $j=\text{ord}_p(d)$, we note that $\text{ord}_p(\tilde{\lambda}(d))=j$ and $p^j$ divides all coordinates of $ \bsr_{d'}-\bsz_p$, and we have $$g_d(\tilde{\bsx})=\frac{1}{\tilde{\lambda}(d)}\sum_{1\leq |\bsi|\leq k} b_{\bsi}(dM\tilde{\bsx}+\bsr_{d'}-\bsz_p)^{\bsi}=p^jf(\tilde{\bsx})+\frac{d}{\tilde{\lambda}(d)}\sum_{|\bsi|=1}b_{\bsi}(M\tilde{\bsx})^{\bsi} +C $$ for some $f\in \Z_p[x_1,\dots,x_r]$ and constant $C$. In particular, modulo $p$, the highest-degree part of $g_d$ is a nonzero multiple of the nonvanishing linear part of $g(\tilde{\bsx}+\tilde{\bsz}_p)$, hence $g_d$ is Deligne modulo $p$. All five properties of $g_d$ are now verified and the proof of Theorem \ref{dimlowthm} is complete.
 \end{proof}
  
\section{Exponential sum estimates} \label{expest} In this final section, we establish the exponential sum estimates claimed in Theorem \ref{standalonethm}, which we then use to deduce (\ref{SmajII}) and (\ref{SminII}). This effort consists primarily of careful multivariate adaptations of the tools used to prove Theorem 2.7 in \cite{ricemax}, but we begin with another foray into varieties over finite fields. 

\subsection{Control over gradient vanishing: Part II} \label{gv2} Since we are sieving away inputs at which the gradient of our polynomial vanishes, but then appealing to Theorem \ref{delmain}, which is a complete exponential sum estimate, it is important for us to have an upper bound on the number of points our sieve might be throwing away. With this in mind, we make the following definition.

\begin{definition} For a field $F$ and $g\in F[x_1,\dots,x_{\ell}]$ we define the \textit{gradient locus} of $g$ to be the variety $$\mathcal{G}_g =\{\bsx\in \mathbb{A}^{\ell}: \grad g(\bsx)=\bszero\}\subseteq \mathbb{A}^{\ell}. $$

\end{definition}

The following proposition establishes firm control over the gradient locus of a Deligne polynomial. 

\begin{proposition}\label{gradcodim}
Suppose $F$ is a field, $\ell \in \N$, and $g \in F[x_1,\ldots,x_\ell]$ with $\deg(g)=k \ge 1$. If $g$ is Deligne, then $\calG_g=\emptyset$ or $\dim \calG_g = 0$.
\end{proposition}

\begin{proof}
First, assume $g$ is not homogeneous. Let $G(x_0,x_1,\ldots,x_\ell)$ be the homogenization of $g$. Thus, we have
	\[
		g(x_1,\ldots,x_\ell) = G(1,x_1,\ldots,x_\ell) \quad\text{and}\quad g^k(x_1,\ldots,x_\ell) = G(0,x_1,\ldots,x_\ell).
	\]
The variety
	\[
		\What := \{G = 0\} \cap \{x_0 = 0\} \subset \bbP^\ell
	\]
is isomorphic to $\{g^k = 0\}$, hence is nonsingular since $g$ is Deligne. Thus, the Jacobian matrix
	\[
		\begin{pmatrix}
		\dfrac{\partial G}{\partial x_0} & \dfrac{\partial G}{\partial x_1} & \cdots & \dfrac{\partial G}{\partial x_\ell}\\
		\\
		1 & 0 & \cdots & 0
		\end{pmatrix}
	\]
has rank $2$ at every point on $\What$. In other words, the system
	\begin{align*}
		G=x_0=\frac{\partial G}{\partial x_1} = \cdots = \frac{\partial G}{\partial x_\ell} &= 0
	\end{align*}
has no solutions in $\bbP^\ell$. The equation $G = 0$ is actually superfluous here; by Euler's theorem on homogeneous functions, we have
	\[
		kG(x_0,x_1,\ldots,x_\ell) = x_0\dfrac{\partial G}{\partial x_0} + x_1\dfrac{\partial G}{\partial x_1} + \cdots + x_\ell\dfrac{\partial G}{\partial x_\ell},
	\]
so the vanishing of $x_0$ and the $x_1$- through $x_\ell$-partials would guarantee the vanishing of $G$. Here we use the fact that the characteristic of $F$ does not divide $k$, as included in the definition of the Deligne property. It follows that the system
	\begin{align*} 
		x_0=\frac{\partial G}{\partial x_1} = \cdots = \frac{\partial G}{\partial x_\ell} &= 0
	\end{align*}
has no solutions in $\bbP^\ell$, so the subvariety of $\bbP^\ell$ defined by
	\begin{equation}\label{eq:gradient_vanishing}
		\frac{\partial G}{\partial x_1} = \cdots = \frac{\partial G}{\partial x_\ell} = 0 
	\end{equation}
is contained in $\{\bsx \in \bbP^\ell \mid x_0 \ne 0\} \cong \bbA^\ell$ and has dimension $0$. But, for $\boldsymbol{\alpha} = (\alpha_1,\ldots,\alpha_\ell) \in \bbA^\ell$, we have
	\[
	\frac{\partial g}{\partial x_i}(\boldsymbol{\alpha}) = \frac{\partial G}{\partial x_i}(1, \boldsymbol{\alpha})
	\]
for all $1 \le i \le \ell$. Thus, $\calG_g$ is (isomorphic to) the zero-dimensional subvariety of $\bbP^\ell$ given by \eqref{eq:gradient_vanishing}, concluding the proof in the case that $g$ is not homogeneous.

Finally, suppose $g$ is homogeneous. Again using Euler's theorem on homogeneous functions, we write
	\[
		kg(x_1,\ldots,x_\ell) = x_1\frac{\partial g}{\partial x_1} + \cdots + x_\ell\frac{\partial g}{\partial x_\ell}.
	\]
Thus, if all partials of $g$ vanish at $\bsx$, then $g(\bsx) = 0$ as well. By hypothesis, $g = g^k$ is smooth, so there are no common zeroes of $g, \frac{\partial g}{\partial x_1}, \ldots, \frac{\partial g}{\partial x_\ell}$ in $\bbP^\ell$, so in $\bbA^\ell$ the only possible common zero is the origin. Therefore, $\calG_g$ contains at most one point.
\end{proof}

Proposition \ref{gradcodim} combines with B\'ezout's Theorem (Lemma \ref{lem:bezout}) to yield the following estimate on the size of the gradient vanishing locus for a Deligne polynomial over a finite field, which yields Lemma \ref{gradconst} as a special case. 
  
\begin{corollary}\label{gradcor} If $\l \geq 1$ and $g\in \F_q[x_1,\dots,x_{\l}]$ is a Deligne polynomial of degree $k\geq 1$, then $|\mathcal{G}_g|$ is bounded by a constant depending only $k$ and $\ell$.
\end{corollary}

\subsection{Major arc estimates} In this section we establish item (\ref{majitem}) of Theorem \ref{standalonethm}. Derivations of asymptotic formulas of this type typically rely on partial summation, so we begin with a multivariate version thereof, proven by induction from the usual formula.
 
\begin{lemma}[Multivariable Partial Summation] \label{mps}

Suppose $\ell\in \N$ and $a:\N^{\ell}\to \C$. Suppose further that $\psi: \R^{\ell}\to \C$ is $C^{\ell}$. For any $X>0$, we have \begin{align*}\sum_{\bsn \in [1,X]^{\ell}} a(\bsn)\psi(\bsn)&= A(X,\dots,X)\psi(X,\dots,X) \\ &\qquad +\sum_{i=1}^{\ell} (-1)^i\sum_{1\leq j_1<\cdots<j_i\leq \ell} \int_{[0,X]^{i}} A(\star)\frac{\partial^i \psi}{\partial x_{j_1}\cdots \partial x_{j_i}}(\star) \ dx_{j_1}\cdots dx_{j_i},\end{align*} where $$A(x_1,\dots,x_{\ell})=\sum_{\bsn \in [1,x_1]\times \cdots \times [1,x_{\ell}]} a(\bsn)$$ and $\star=(X,\dots, x_{j_1},\dots ,x_{j_{i}}, \dots, X),$ with $x_{j_1},\dots,x_{j_i}$ plugged into coordinate positions $j_1,\dots,j_i$ and all other coordinates evaluated at $X$.

\end{lemma}

\begin{proof} We induct on $\ell$. The base case $\ell=1$ is the usual partial summation formula $$\sum_{1\leq n \leq X} a(n)\psi(n)= A(X)\psi(X)-\int_0^X A(x)\psi'(x) \ dx. $$ Fix $\ell\geq 2$ and assume the formula holds for $\ell-1$. Defining some notation before proceeding, let  $$\tilde{A}(x_1,\dots,x_{\ell-1},n_{\ell})=\sum_{\bsn \in [1,x_1]\times \cdots \times [1,x_{\ell-1}]} a(\bsn,n_{\ell}), $$ let $$\tilde{I}(j_1,\dots,j_i,n_{\ell})=\int_{[0,X]^{i}} \tilde{A}(\star,n_{\ell})\frac{\partial^i \psi}{\partial x_{j_1}\cdots \partial x_{j_i}}(\star,n_{\ell}) \ dx_{j_1}\cdots dx_{j_i}, $$ and let $$I(j_1,\dots,j_i)=\int_{[0,X]^{i}} A(\star)\frac{\partial^i \psi}{\partial x_{j_1}\cdots \partial x_{j_i}}(\star) \ dx_{j_1}\cdots dx_{j_i}, $$ where $A$ and $\star$ are as defined in the statement of the lemma. By our inductive hypothesis, we have \begin{align*} \sum_{\bsn \in [1,X]^{\ell}} a(\bsn)\psi(\bsn)&= \sum_{1\leq n_{\ell}\leq X} \sum_{\bsn \in [1,X]^{\ell-1}} a(\bsn, n_{\ell})\psi(\bsn, n_{\ell}) \\ &=\sum_{1\leq n_{\ell}\leq X} \Big( \tilde{A}(X,\dots,X,n_{\ell})\psi(X,\dots,X,n_{\ell}) +\sum_{i=1}^{\ell-1} (-1)^i\sum_{1\leq j_1<\cdots<j_i\leq \ell-1} \tilde{I}(j_1,\dots,j_i,n_{\ell})\Big) . \end{align*} We now apply the standard single-variable formula to the first term and each individual integral, yielding \begin{equation}\label{fterm} \sum_{1\leq n_{\ell}\leq X} \tilde{A}(X,\dots,n_{\ell})\psi(X,\dots,n_{\ell}) = A(X,\dots,X)\psi(X,\dots,X)-\int_0^X A(X,\dots,x_{\ell})\frac{\partial\psi}{\partial x_{\ell}}(X,\dots,x_{\ell}) dx_{\ell},\end{equation} and \begin{align*}\sum_{1\leq n_{\ell}\leq X} & (-1)^i\tilde{I}(j_1,\dots,j_i,n_{\ell}) \\ &= (-1)^i\int_{[0,X]^{i}} \left(A(\star,X)\frac{\partial^i \psi}{\partial x_{j_1}\cdots \partial x_{j_i}}(\star,X)-\int_0^X A(\star,x_{\ell})\frac{\partial^{i+1} \psi}{\partial x_{j_1}\cdots \partial x_{j_i}\partial x_{\l}}(\star,x_{\ell})dx_{\ell}\right)\ dx_{j_1}\cdots dx_{j_i} \\ &=(-1)^i I(j_1,\dots,j_i)+(-1)^{i+1}I(j_1,\dots,j_i,\ell). \end{align*} Summing this final expression over $1\leq i \leq \ell-1$ and over all choices of $1\leq j_1<\cdots<j_i\leq \ell-1$ accounts for all required terms with $1\leq i \leq \ell$ and $1\leq j_1<\cdots<j_i\leq \ell$, with the single exception of $i=1$ and $j_1=\ell$, which is precisely the integral present in (\ref{fterm}), and the induction is complete. \end{proof} 

We use Lemma \ref{mps} and the same calculation as in Proposition \ref{brunprop} to establish our asymptotic formula for sieved multivariate exponential sums near rationals with small denominator. 

\begin{lemma}\label{Sasym} Suppose $\ell,k\in \N$, $g(\bsx)=\sum_{|\bsi|\leq k} a_{\bsi} \bsx^{\bsi} \in \Z[x_1,\dots,x_{\ell}]$, and let $J=\sum_{|\bsi|\leq k} |a_{\bsi}|$. If $X,Y > 0$, $a,q\in \N$, and $\alpha=a/q+\beta$, then \begin{align*}\sum_{\bsn \in [1,X]^{\ell} \cap W(Y)}e^{2\pi i g(\bsn)\alpha}&=q^{-\ell} \prod_{\substack{ p\leq Y \\ p^{\gamma(p)}\nmid q}}\left(1-\frac{j(p)}{p^{\gamma(p)\ell}}
\right)\sum_{\boldsymbol{s}\in \{0,\dots,q-1\}^{\ell} \cap W^{q}(Y)}e^{2\pi i g(\boldsymbol{s})a/q}\int_{[0,X]^{\ell}}e^{2\pi i g(\bsx)\beta}d\bsx\\\\&\qquad + O_{k,\l}\left(qE(1+JX^{k}|\beta|)^{\ell}\right)  ,\end{align*} where $E$ is as in Proposition \ref{brunprop}.

\end{lemma}

\begin{proof} We begin by noting that for any $a,q \in \N$ and $0\leq x_1,\dots,x_{\ell} \leq X$, letting $$B=[1,x_1]\times\cdots\times [1,x_{\ell}],$$ we have
\begin{align*}T(x_1,\dots,x_{\ell})&:=\sum_{\bsn \in B\cap W(Y)}e^{2\pi i g(\bsn)a/q}\\ &\phantom{:}=\sum_{\boldsymbol{s}\in \{0,\dots,q-1\}^{\ell}} e^{2\pi i g(\boldsymbol{s})a/q} \left| \left\{\bsn \in B \cap W(Y) :  \bsn \equiv \bss \ (\text{mod }q) \right\} \right|.
\end{align*}
For $s\in W^q(Y)$ we have by the same calculation as Proposition \ref{brunprop} that \begin{equation*} \label{count} \left| \left\{\bsn \in B \cap W(Y) :  \bsn \equiv \bss \ (\text{mod }q) \right\} \right|=\frac{x_1\cdots x_{\ell}}{q^{\ell}}\prod_{\substack{p\leq Y\\ p^{\gamma(p)} \nmid q}} \left(1-\frac{j(p)}{p^{\gamma(p)\ell}} \right)+E/q^{\l-1} ,\end{equation*} where $E$ is as in Proposition \ref{brunprop}, whereas for $s\notin W^q(Y)$ the set is empty.

\noindent  Therefore, \begin{equation} \label{Tx} T(x_1,\dots,x_{\ell})=\frac{x_1\cdots x_{\ell}}{q^{\ell}}\prod_{\substack{p\leq Y\\ p^{\gamma(p)} \nmid q}} \left(1-\frac{j(p)}{p^{\gamma(p)\ell}} \right)\sum_{\boldsymbol{s}\in \{0,\dots,q-1 \}^{\ell}\cap W^q(Y)} e^{2\pi i g(\boldsymbol{s})a/q}+O\left(qE\right). \end{equation} Letting $\psi(\bsn)=e^{2\pi i g(\bsn) \beta}$, we now decompose our sum as $$\sum_{\bsn \in [1,X]^{\ell} \cap W(Y)}e^{2\pi i g(\bsn)\alpha}=\sum_{\bsn \in [1,X]^{\ell}} \left(1_{W(Y)}(\bsn)e^{2\pi i g(\bsn)a/q}\right) \psi(\bsn)  $$ and apply Lemma \ref{mps}, yielding 
\begin{align*} \sum_{\bsn \in [1,X]^{\ell} \cap W(Y)}e^{2\pi i g(\bsn)\alpha}&=T(X,\dots,X)\psi(X,\dots,X) \\ &\qquad+\sum_{m=1}^{\ell} (-1)^m\sum_{1\leq j_1<\cdots<j_m\leq \ell} \int_{[0,X]^{m}} T(\star)\frac{\partial^m \psi}{\partial x_{j_1}\cdots \partial x_{j_m}}(\star) \ dx_{j_1}\cdots dx_{j_m},
\end{align*} where $\star$ is as in Lemma \ref{mps}. Substituting (\ref{Tx}) gives the main term \begin{align*}&q^{-\ell}\prod_{\substack{p\leq Y\\ p^{\gamma(p)} \nmid q}} \left(1-\frac{j(p)}{p^{\gamma(p)\ell}} \right)\sum_{\boldsymbol{s}\in \{0,\dots,q-1 \}^{\ell}\cap W^q(Y)} e^{2\pi i g(\boldsymbol{s})a/q} \Big( X^{\ell}\psi(X,\dots,X) \\ &\qquad + \sum_{m=1}^{\ell} (-1)^m\sum_{1\leq j_1<\cdots<j_m\leq \ell} X^{\ell-m}\int_{[0,X]^{m}} x_{j_1}\cdots x_{j_m}\frac{\partial^m \psi}{\partial x_{j_1}\cdots \partial x_{j_m}}(\star) \ dx_{j_1}\cdots dx_{j_m}\Big). \end{align*} By iteratively applying integration by parts, this equals \begin{equation*}q^{-\ell}\prod_{\substack{p\leq Y\\ p^{\gamma(p)} \nmid q}} \left(1-\frac{j(p)}{p^{\gamma(p)\ell}} \right) \sum_{\boldsymbol{s}\in \{0,\dots,q-1 \}^{\ell}\cap W^q(Y)} e^{2\pi i g(\boldsymbol{s})a/q} \int_{[0,X]^{\ell}} \psi(\bsx) d\bsx, \end{equation*} as desired. It remains to bound the error term that results from our substitution of (\ref{Tx}). This error term is the sum of a first term of order $qE$ and   $2^{\ell}-1$ terms of the form \begin{equation*}qE\left(\int_{[0,X]^{m}} \frac{\partial^m \psi}{\partial x_{j_1}\cdots \partial x_{j_m}}(\star) \ dx_{j_1}\cdots dx_{j_m} \right). \end{equation*} Iteratively applying the product rule, we see that $\frac{\partial^m \psi}{\partial x_{j_1}\cdots \partial x_{j_m}}$ is the sum of less than $m!$ terms bounded in absolute value by $(2\pi k^{m}J|\beta|)^jX^{jk-m}$ for some $1\leq j \leq m$. In particular, each integral is bounded by $$\ell!\max_{1\leq j \leq \ell} (2\pi k^{\ell}JX^k|\beta|)^j \leq \ell!(1+2\pi k^{\ell}JX^k|\beta|)^{\ell}, $$ and the error bound follows.\end{proof}

\subsection{Local cancellation} In this section, we apply Theorem \ref{delmain} to establish the necessary cancellation in our sieved local exponential sums, yielding item (\ref{locitem}) in Theorem \ref{standalonethm}. We begin by invoking a multivariate version of Hensel's Lemma that allows us to reduce to the case of prime moduli. This statement in particular follows from Theorem~1.1 of \cite{Conrad}.

\begin{lemma}[Multivariable Hensel's Lemma] \label{hensel} Suppose $\ell\in \N$, $g\in \Z[x_1,\dots,x_{\ell}]$, $p$ is prime, $\bsn \in \Z^{\l}$, and $\gamma,v\in \N$ with $v\geq 2\gamma-1$. If $$g(\bsn)\equiv 0 \ \textnormal{mod} \ p^{2\gamma-1}$$ and $\grad g (\bsn) \not\equiv \bszero$ \textnormal{mod} $p^{\gamma}$, then there exists $\bsm\in \Z^{\l}$ with $g(\bsm) \equiv 0$ \textnormal{mod} $p^v$.
\end{lemma}

We now prove the following multivariate generalization of Lemma 4.3 in \cite{ricemax}.

\begin{lemma}\label{locgen} Suppose $\ell\in \N$, $g\in \Z[x_1,\dots, x_{\l}]$ with $\deg(g)=k\geq 2$, and $Y>0$. If $q\in \N$ has prime factorization $q=p_1^{v_1}\cdots p_r^{v_r}$ with $p_1<\cdots< p_t\leq Y < p_{t+1}< \cdots < p_r$, and $(a,q)=1$, then $$\left| \sum_{\boldsymbol{s}\in \{0,\dots,q-1\}^{\ell} \cap W^{q}(Y)}e^{2\pi i g(\boldsymbol{s})a/q} \right| \leq C_1 \prod_{i=1}^t \left((k-1)^{\l}p_i^{\l/2}+j(p_i)\right)\prod_{i=t+1}^r C_2(v_i+1)^{\ell} p_i^{v_i(\l-1/k)}, $$ where $C_2=C_2(k)$ and $C_1$ depends only on the moduli at which $\grad g$ identically vanishes and the primes $p\leq Y$ dividing $q$ modulo which $g$ is not Deligne. Further, the sum is $0$ if $v_i\geq 2\gamma(p_i)$ for some $1\leq i \leq t$.

\end{lemma}

\begin{proof} Factor $q=p_1^{v_1}\cdots p_r^{v_r}$ as in the lemma. By the Chinese Remainder Theorem, we have \begin{equation*} \sum_{\boldsymbol{s}\in \{0,\dots,q-1\}^{\ell} \cap W^{q}(Y)}e^{2\pi i g(\bss)a/q}= \prod_{m=1}^{r} \sum_{\boldsymbol{s}\in \{0,\dots,p_m^{v_m}-1\}^{\ell} \cap W^{p_m^{v_m}}(Y)}e^{2\pi i g(\bss)a_m/p_m^{v_m}}, \end{equation*} where $a_1,\dots,a_r$ are the unique residues satisfying $a/q \equiv a_1/p_1^{v_1}+\cdots+a_{r}/p_r^{v_r} \ \text{mod }1. $

\noindent Suppose $p^v=p_m^{v_m}$ with $\gamma(p)>1$ and $v<2\gamma(p)$. By definition of $\gamma$, $\grad g$ identically vanishes modulo $p^{\gamma(p)-1}$. Since $p^{2\gamma(p)-1}\leq p^{3(\gamma(p)-1)}$, we can bound $p^{v}$ by the cube of a modulus at which $\grad g$ identically vanishes, trivially bound the corresponding sum, and absorb it into the constant $C_1$ in the conclusion of the lemma. 

\noindent Next suppose $p^v=p_m^{v_m}$ with $p\leq Y$ and $v=\gamma(p)=1$. Recalling that $j(p)$ is the number of zeros of $\grad g$ modulo $p$  and applying Theorem \ref{delmain}, we have for $p\nmid b$ that  $$\left|\sum_{\bss \in \{0,\dots,p-1\}^{\ell}\cap W^{p}(Y)}e^{2\pi i g(\bss)b/p}\right| \leq (k-1)^{\l}p^{\ell/2}+j(p), $$ provided $g$ is Deligne modulo $p$, and the remaining such primes are absorbed into $C_1$. 

\noindent Now suppose that $p^v=p_m^{v_m}$ with $p\leq Y$ and $v\geq 2\gamma(p)$, and let $w=2\gamma(p)-1$. If $\bss \in \{0,\dots, p^v-1\}^{\ell}$ and $\tilde{\bss}$ is the reduced residue class of $\bss$ modulo $p^{w}$, then we have that $g(\bss)\equiv p^{w}t+g(\tilde{\bss}) \ (\text{mod }p^v)$ for some $0\leq t\leq p^{v-w}-1$. Conversely,  if $\tilde{\bss}\in \{0,\dots, p^w-1\}^{\ell}$ with  $\grad g(\tilde{\bss})\not\equiv \bszero \  (\text{mod }p^{\gamma(p)})$, then for every $0\leq t \leq p^{v-w}-1$, Lemma \ref{hensel} applied to the polynomial $g(\bsx)-(p^{w}t+g(\tilde{\bss}))$ yields $\bss \in \{0,\dots,p^v-1\}^{\ell}$ with $g(\bss)\equiv p^{w}t+g(\tilde{\bss}) \ (\text{mod }p^v)$.
 
\noindent In other words, the map $F$ on $\Z/p^{v-w}\Z$ defined by $g(p^{w}t+\tilde{\bss})\equiv p^{w}F(t)+g(\tilde{\bss}) \ (\text{mod }p^v)$ is a bijection.  In particular, if $p\nmid b$, then \begin{align*}\sum_{\bss\in \{0,\dots, p^v-1\}^{\ell} \cap W^{p^v}(Y)}e^{2\pi i g(\bss)b/p^v} &= \sum_{\substack{\tilde{\bss}\in \{0,\dots,p^w-1\}^{\ell} \\ \grad g(\tilde{\bss})\not\equiv \bszero \ (\text{mod }p^{\gamma(p)})}} \sum_{t=0}^{p^{v-w}-1}e^{2\pi \i g(p^{w}t+\tilde{\bss})b/p^v}\\  &=\sum_{\substack{\tilde{\bss}\in \{0,\dots,p^w-1\}^{\ell} \\ \grad g(\tilde{\bss})\not\equiv \bszero \ (\text{mod }p^{\gamma(p)})}} \sum_{t=0}^{p^{v-w}-1}e^{2\pi \i \left(p^{w}t+g(\tilde{\bss})\right)b/p^v} \\ &=0,\end{align*} where the last equality is the fact that the sum in $t$ runs over the full collection of $p^{v-w}$-th roots of unity. 

\noindent Finally, suppose $p^v=p_m^{v_m}$ with $p> Y$. We note that $W^{p^v}(Y)=\N$ and we only exploit cancellation in a single variable. To this end, for each $\tilde{\bss}=(s_2,\dots,s_{\ell})\in \{0,\dots,p^v-1\}^{\ell-1}$, we define $\tilde{g}$ by $\tilde{g}(x)=g(x,\tilde{\bss})$.  Utilizing the  standard single-variable complete sum estimate (see \cite{Chen} for example), we have for $b\nmid p$ that  \begin{align*}\left|\sum_{\bss\in\{0,\dots,p^v-1\}^{\ell}}e^{2\pi i g(\bss)b/p^v}\right| &\leq  \sum_{\tilde{\bss}\in \{0,\dots,p^v-1\}^{\ell-1}} \left| \sum_{s=0}^{p^v-1} e^{2\pi i \tilde{g}(s)b/p^v} \right| \\ &\ll_k  p^{v(1-1/k)}\sum_{\tilde{\bss}\in \{0,\dots,p^v-1\}^{\ell-1}}\gcd(\text{cont}(\tilde{g}),p^v)^{1/k}. \end{align*} 
To analyze the remaining sum, we note that at the expense of the term $\gcd(\text{cont}(g),p^v)^{1/k}$ in our final estimate, we can cancel factors of $p$ from the coefficients of $g$ and assume that $p\nmid \text{cont}(g)$. In this case, suppose $a_{\bsi}=a_{i_1,\dots,i_{\ell}}$ with $0<|\bsi|\leq k$ is a coefficient of $g$, corresponding to $x_1^{i_1}\cdots x_{\l}^{i_{\l}}$, that is not divisible by $p$. Further, assume that $i_1>0$, as if $i_1=0$ then we could just relabel our coordinates. In this case, for each $0\leq w \leq v$, $\gcd(\text{cont}(\tilde{g}),p^v)=p^w$ only if $p^w\mid s_2^{i_2}\cdots s_{\ell}^{i_{\l}}$, so in particular $p^{\lceil w/k \rceil}\mid s_2\cdots s_{\ell}$, which occurs for fewer than $(w+1)^{\ell-1}p^{v(\ell-1)-w/k}$ choices of $\tilde{\bss}$. In particular, \begin{align*}\sum_{\tilde{s}\in \{0,\dots,p^v-1\}^{\ell-1}}\gcd(\text{cont}(\tilde{g}),p^v)^{1/k}&\leq \gcd(\text{cont}(g),p^v)^{1/k} \sum_{w=0}^{v} (w+1)^{\ell-1}p^{v(\ell-1)-w/k}p^{w/k} \\ &\leq (v+1)^{\ell}\gcd(\text{cont}(g),p^v)^{1/k}p^{v(\ell-1)}. \end{align*}  The $\gcd(\text{cont}(g),p^v)^{1/k}$ term can be absorbed into $C_1$, and the remaining bound on the exponential sum modulo $p^v$ is a constant depending on $k$ times $p^{v(1-1/k)}(v+1)^{\ell} p^{\ell(v-1)}=(v+1)^{\ell}p^{v(\ell-1/k)}$, as required. Having accounted for all prime divisors of $q$, the proof is complete.
 \end{proof} 

Lemma \ref{locgen} combines with Lemma \ref{gradconst} as well as the estimates $\prod_{p\mid q}\left(1+\frac{C}{p} \right)\leq (q/\phi(q))^C$ and $\prod_{p\mid q}\left(1+\frac{C}{p^{3/2}} \right)\ll_C 1$  to yield item (\ref{locitem}) of Theorem \ref{standalonethm}, restated below. 

\begin{corollary}\label{qcor} If $\ell\geq 2$, $g\in \Z[x_1,\dots, x_{\l}]$ with $\deg(g)=k\geq 2$, and $(a,q)=1$, then $$\left| \sum_{\boldsymbol{s}\in \{0,\dots,q-1\}^{\ell} \cap W^{q}(Y)}e^{2\pi i g(\boldsymbol{s})a/q} \right| \leq C_1\begin{cases} (k-1)^{\l \omega(q)}\Phi(q,\l) q^{\ell/2} &\text{if }q\leq Y \\ C_2^{\omega(q)}\tau(q)^{\l}q^{\ell-1/k} &\text{for all }q \end{cases},$$ where $C_2=C_2(k)$, $\Phi(q,2)=(q/\phi(q))^{C_2}$, $\Phi(q,\l)\ll_{k,\l} 1$ for $\l\geq 3$, and $C_1$ depends only on the moduli at which $\grad g$ identically vanishes and the primes $p\leq Y$ dividing $q$ modulo which $g$ is not Deligne.

\end{corollary}

\subsection{Oscillatory integral estimate} In order to establish (\ref{SminII}) in the case that $\alpha$ is close, but not too close, to a rational with very small denominator, we need to control the oscillatory integral in the asymptotic formula given by Lemma \ref{Sasym}. To achieve this, we invoke the following standard estimate, given for example in Lemma 2.8 of \cite{vaughan}. 

\begin{lemma}[Van der Corput's Lemma] \label{vdcl} If  $X>0$, $\beta\neq 0$, $k\in \N$, and $g\in \Z[x]$ with $\deg(g)=k$, then $$\left|\int_0^X e^{2\pi i g(x)\beta} dx \right| \ll |\beta|^{-1/k}. $$

\end{lemma}

Utilizing Lemma \ref{vdcl} to exploit cancellation in a single variable, then trivially bounding the integral in the remaining variables, we have the following bound for the integral in the conclusion of Lemma \ref{Sasym}.

\begin{corollary} \label{vdccor} If $X>0$, $\beta\neq 0$, $k,\ell\in \N$, and $g\in \Z[x_1,\dots,x_{\l}]$ with $\deg(g)=k$, then $$\left|\int_{[0,X]^{\ell}}e^{2\pi i g(\bsx)\beta}d\bsx\right| \ll \min\{ X^{\ell}, X^{\ell-1}|\beta|^{-1/k}\}.$$

\end{corollary}

\subsection{Minor arc estimates}
In an effort to establish item (\ref{minitem}) of Theorem \ref{standalonethm}, we begin by invoking a variation of the most traditional minor arc estimate, Weyl's Inequality.

\begin{lemma}[Lemma 3, \cite{CLR}] \label{weyl2}  Suppose $k\in \N$, $g(x)=a_0+a_1x+\cdots+a_{k}x^{k}$ with $a_0\dots,a_k \in \R$ and $a_{k} \in \N$. If $X>0$, $a,q\in \N$ with $(a,q)=1$, and $|\alpha-a/q|<q^{-2}$, then 
$$ \left|\sum_{n=1}^X e^{2\pi \i g(n)\alpha} \right| \ll_{k} X \left(a_{k}\log^{k^2}(a_{k}qX)\left(q^{-1}+X^{-1}+\frac{q}{a_kX^k}\right) \right)^{2^{-k}}.$$
\end{lemma}

We now carefully adapt Lemma \ref{weyl2} to our particular sieve, and to the multivariate setting, though as in Corollary \ref{vdccor}, we ultimately only exploit cancellation in a single variable. 

\begin{lemma}\label{weyl3} Suppose $k,\ell\in \N$ and $g(\bsx)=\sum_{|\bsi|\leq k} a_{\bsi} \bsx^{\bsi} \in \Z[x_1,\dots,x_{\ell}]$ with $\deg(g)=k$. Suppose further that $X,Y,Z\geq 2$, $YZ\leq X$, and $a,q\in \N$ with $(a,q)=1$, and let $J=\sum_{|\bsi|\leq k} |a_{\bsi}|$. If $|\alpha-a/q|<q^{-2}$, then $$\left|\sum_{\bsn \in [1,X]^{\ell} \cap W(Y)} e^{2\pi \i g(\bsn)\alpha} \right| \ll_{k,\ell} \textnormal{cont}(g)^6(\log Y)^{ek} X^{\ell}\left(e^{-\frac{\log Z}{\log Y}}+\left(J\log^{k^2}(JqX)\left(q^{-1}+\frac{Z}{X}+\frac{qZ^k}{X^k}\right) \right)^{2^{-k}} \right). $$
\end{lemma}

\begin{proof} Suppose $k,\ell\in \N$ and $g(\bsx)=\sum_{|\bsi|\leq k} a_{\bsi} \bsx^{\bsi} \in \Z[x_1,\dots,x_{\ell}]$ with $\deg(g)=k$. We begin by conducting an invertible (over $\Z$) change of variables to reduce to the case where the $x_1^k$ coefficient $a_{(k,0,\dots,0)}$ is nonzero. To this end, consider the polynomial $\tilde{g}\in \Z[x_2,\dots,x_{\ell}]$ defined by $\tilde{g}(x_2,\dots,x_{\ell})=g^k(1,x_2,\dots,x_{\ell})$, where $g^k$ denotes the top degree homogeneous part of $g$, noting that $\tilde{g}$ is not identically zero. Let $(c_2,\dots,c_{\ell})\in \{0,1,\dots,k\}^{\ell-1}$ be such that $\tilde{g}(c_2,\dots,c_{\ell})\neq 0$. 

\noindent As an aside, the existence of such a ``small integer non-root'' of a general nonzero multivariate polynomial $F\in \Z[x_1,\dots,x_j]$ can be shown via induction, which we sketch here. The base case $j=1$ corresponds to nonzero univariate polynomials, which have at most $k$ roots, hence at least one non-root in $\{0,1,\dots,k\}$. Then, for higher degrees, fix one variable that appears at least once in $F$ (without loss of generality, assume $x_1$ appears at least once), let $d$ be the degree of $F$ as a polynomial in $x_1$ only, and let $\tilde{F}(x_2,\dots,x_j)$ be the polynomial of degree at most $k-d$ that forms the $x_1^d$ coefficient. By the inductive hypothesis, we can choose $(m_2,\dots,m_j)\in \{0,\dots,k\}^{j-1}$ such that $\tilde{F}(m_2,\dots,m_j)\neq 0$. Then, $F(x_1,m_2,\dots,m_j)$ is a nonzero degree-$d$ polyomial in $x_1$, which has a non-root in $\{0,\dots,k\}$, completing the induction.

\noindent Back to the proof at hand, we see that the change of variables $x_1=y_1$ and $x_j=y_j+c_jy_1$ for $2\leq j \leq \ell$ yields a $y_1^k$ coefficient of $\tilde{g}(c_2,\dots,c_{\ell})\neq 0$. Let $M$ denote the $\ell\times \ell$ matrix satisfying $M\bsx =  \boldsymbol{y}$ corresponding to the described change of variables, and let $f(y_1,\dots,y_{\l})=\sum_{|\bsi|\leq k} b_{\bsi} y^{\bsi}$ be the polynomial satisfying $f(\boldsymbol{y})=g(M^{-1}\boldsymbol{y})$. By taking the complex conjugate of the relevant exponential sum if necessary, we can assume that $b=b_{(k,0,\dots,0)}>0$. Further, the effect of the transformation on the size of this coefficient is well-controlled, in that $b\ll_{k,l} J$.  

\noindent Let $T=M([1,X]^{\ell})$, so $$ \sum_{\bsn \in [1,X]^{\ell} \cap W(Y)}e^{2\pi \i g(\bsn)\alpha}=\sum_{\bsn \in T \cap W(Y)}e^{2\pi \i f(\bsn)\alpha},$$ where $W(Y)$ is defined on each side in terms of the corresponding polynomial. 

\noindent Let $\tilde{T}$ denote the projection of $T$ onto the last $\ell-1$ coordinates, noting that $|\tilde{T}|\leq (2kX)^{\ell-1}$ due to the details of our change of variables. For each fixed $\tilde{\bsn}=(n_2,\dots,n_{\ell})\in \N^{\ell-1}$, we let $I=\{n\in \N: (n,\tilde{\bsn})\in T\}$, which is an interval of integers of length at most $X$, we let $\tilde{W}(Y)=\{n\in \N: (n,\tilde{\bsn})\in W(Y)\}$, and we let $\tilde{f}(x)=f(x,\tilde{\bsn})$. We see trivially that \begin{equation}\label{T} \left|\sum_{\bsn \in T \cap W(Y)} e^{2\pi \i f(\bsn)\alpha} \right| \leq (2kX)^{\ell-1}\max_{\tilde{\bsn}\in \tilde{T}} \left|\sum_{n \in I \cap \tilde{W}(Y)} e^{2\pi \i \tilde{f}(n)\alpha} \right|. \end{equation} 

\noindent We now proceed with $\tilde{\bsn}=(n_2,\dots,n_{\ell})$ fixed, and we define $L$ and $m$ so that $I=[m,L+m]$, so in particular $L\leq X$. All subsequent conclusions will be independent of $\tilde{\bsn}$. Let  $P$ be the set of products $p_1^{\gamma(p_1)}\cdots p_s^{\gamma(p_s)}$ for primes $p_1<\cdots<p_s\leq Y$, let $P_1$ denote the set of elements of $P$ that are at most $Z$, and let $P_2$ denote the set of elements of $P$ that are greater than $Z$. 

\noindent By inclusion-exclusion, we have \begin{equation}\label{P1P2} \left|\sum_{n \in I \cap \tilde{W}(Y)} e^{2\pi \i \tilde{f}(n)\alpha}\right| =\left|\sum_{D\in P} (-1)^{\omega(D)} \sum_{\substack{n\in I \\ \grad f(n,\tilde{\bsn})\equiv \bszero \ (\text{mod }D)}} e^{2\pi \i \tilde{f}(n)\alpha}\right|,
\end{equation}
where $\omega(D)$ is the number of distinct prime factors of $D$. For $D\in P_1$, we use the fact that the set of $n$ for which $\grad f(n,\tilde{\bsn})\equiv \bszero \ (\text{mod }D)$ is contained in the set of $n$ for which $\tilde{f}'(n)\equiv 0 \ (\text{mod }D)$. Noting that $\tilde{f}'$ can have at most $k$ roots modulo any prime at which it does not identically vanish, we have 
\begin{align*}\left|\sum_{D\in P_1} (-1)^{\omega(D)} \sum_{\substack{n\in I \\ \grad f(n,\tilde{\bsn})\equiv \bszero \ (\text{mod }D)}} e^{2\pi \i \tilde{f}(n)\alpha}\right| \ll_k (\text{cont}(g))^2 \sum_{D \in P_1} k^{\omega(D)} \max_{0\leq c \leq D}\left| \sum_{n=0}^{L/D} e^{2\pi \i \tilde{f}(Dn+m+c)\alpha}\right|,
\end{align*} where the $\text{cont}(g)^2$ term accounts for the primes $p$ for which $\gamma(p)>1$ by Proposition \ref{idzero}. Further, we see from Lemma \ref{weyl2} and the estimate $1\leq b\ll_{k,l} J$ that \begin{align*}\sum_{D \in P_1} k^{\omega(D)} \max_{0\leq c \leq D}\left| \sum_{n=0}^{L/D} e^{2\pi \i \tilde{f}(Dn+m+c)\alpha}\right| &\ll_{k,l} \sum_{D\in P_1} k^{\omega(D)} \frac{L}{D} \left(b\log^{k^2}(bqL)\left(q^{-1}+\frac{D}{L}+\frac{qD^k}{bL^k}\right) \right)^{2^{-k}} \\ &\ll_{k,l} X \left(J\log^{k^2}(JqX)\left(q^{-1}+\frac{Z}{X}+\frac{qZ^k}{X^k}\right) \right)^{2^{-k}}\sum_{D\in P_1} \frac{k^{\omega(D)}}{D} \\ &\ll_{k,l} X(\log Y)^k\left(J\log^{k^2}(JqX)\left(q^{-1}+\frac{Z}{X}+\frac{qZ^k}{X^k}\right) \right)^{2^{-k}},\end{align*} where the last inequality uses that if $C>0$, then \begin{equation}\label{Z}\sum_{D\in P} \frac{C^{\omega(D)}}{D} = \prod_{p\leq Y} \left(1+\frac{C}{p^{\gamma(p)}}\right) \leq   \prod_{p\leq Y} \left(1+\frac{C}{p}\right) \ll (\log Y)^C.\end{equation} This combines with (\ref{T}) to close the book on the contributions to (\ref{P1P2}) from $P_1$. It remains to account for the contribution to (\ref{P1P2}) from $P_2$. Because $P_2$ has so many elements, it is crucial for us to exploit the cancellation provided by the term $(-1)^{\omega(D)}$. 

\noindent To this end, for a  fixed $n\in I$, let $P^n=\{D\in P:  \grad f(n,\tilde{\bsn})\equiv \bszero \ (\text{mod }D)\}$, and let $P^n_2=P^n\cap P_2$. The only issue is the possibility that way more elements of $P^2_n$ have an even number of prime factors than odd, or vice versa, which we show below does not happen.

\noindent Let $q$ be the largest prime power of the form $p^{\gamma(p)}$ with $p\leq Y$, and let $q_n$ be the largest such prime power lying in $P^n$,  noting that $q_n\leq q\ll_k \cont(g)Y$ by Proposition \ref{idzero}. \noindent Let $A$ denote the set of elements of $P^n$ that have an even number of prime factors, let $B$ denote the set of elements of $P^n$ that have odd number of prime factors, and let $A'$ and $B'$, respectively, denote the same for elements of $P^n_2$. The quantity we need control of is $\left||A'|-|B'|\right|$. 

\noindent  Let $A_1$ be the elements of $A$ that are greater than $Z$ and not divisible by $q_n$, and let $A_2$ be the elements of $A$  that are greater than $q_nZ$ and divisible by $q_n$. Likewise define $B_1$ and $B_2$. The map $D\to q_nD$ defines an injection from $A_1$ to $B_2$, while the map $D\to D/q_n$ defines an injection from $A_2$ to $B_1$. Letting $A_3$ denote all the elements of $A$ greater than $q_nZ$, we have $$|A_3|\leq |A_1|+|A_2| \leq |B_1|+|B_2| \leq |B'|. $$ Symmetrically, we have $|B_3|\leq |A'|$. Finally, letting $A_4$ and $B_4$ denote the elements of $A'$ and $B'$ satisfying $Z<D\leq q_nZ$, we have $|A'|=|A_3|+|A_4|\leq |B'|+|A_4|$ and similarly $|B'|\leq |A'|+|B_4|$, so the magnitude of $|A'|-|B'|$ is bounded above by $|A_4|+|B_4|$, which is the size of the set $\bar{P^n}$ of elements of $P^n$ satisfying $Z < D \le q_nZ$.

\noindent We now see
\begin{align*} \left|\sum_{D\in P_2} (-1)^{\omega(D)} \sum_{\substack{n\in I \\ \grad f(n,\tilde{\bsn})\equiv \bszero \ (\text{mod }D)}} e^{2\pi i \tilde{f}(n)\alpha} \right|& =\left|\sum_{n\in I} e^{2\pi i \tilde{f}(n)\alpha}\sum_{D\in P^n_2} (-1)^{\omega(D)}\right|  \\ & \leq \sum_{n\in I} |\bar{P^n}| \\ &= \sum_{\substack{D\in P \\ Z< D \leq qZ}} |\{n\in I : \grad f(n,\tilde{\bsn})\equiv \bszero \ (\text{mod }D)\}| \\ &\ll_k (\text{cont}(g))^2 \sum_{\substack{D\in P \\ Z< D \leq qZ}} k^{\omega(D)}\left( \frac{L}{D}+1 \right) \\ & \ll (\text{cont}(g))^3 X  \sum_{\substack{D\in P \\ D> Z}} \frac{k^{\omega(D)}}{D},
\end{align*}
provided $YZ \leq X$.  If $D\in P$ with $D>Z$, then, since $D\ll_k \text{cont}(g)^2Y^{\omega(D)}$ and $Y\geq 2$, we know that  \begin{equation}\label{logQY} \text{cont}(g)^3e^{\omega(D)-\frac{\log Z}{\log Y}} \gg_k 1. \end{equation} 
Finally, (\ref{Z}) and (\ref{logQY}) imply \begin{align*} \sum_{\substack{D \in P \\ D> Z}}\frac{k^{\omega(D)}}{D} &\ll_k \text{cont}(g)^3 e^{-\frac{\log Z}{\log Y}} \sum_{D\in P} \frac{(ek)^{\omega(D)}}{D} \\& \ll \text{cont}(g)^3 e^{-\frac{\log Z}{\log Y}} (\log Y)^{ek}, \end{align*} and the lemma follows.
\end{proof} 

We now conclude our discussion by combining the tools developed in this section to establish (\ref{SmajII}) and (\ref{SminII}), thus completing the proof of Theorem \ref{more}.

\subsection{Proof of (\ref{SmajII}) and (\ref{SminII})} We return to the  proof of Lemma \ref{L2I} in Section \ref{massproof}, recalling all assumptions,  notation, and fixed parameters. We let $Z=N^{c_0}$, and we let $J$ denote the sum of the absolute value of the coefficients of $h_d$, noting that \begin{equation}\label{Jb} J\ll_h d^k \leq Z^k. \end{equation} Fixing $\alpha\in \T$, the pigeonhole principle guarantees the existence of $1\leq q \leq M^k/Z^{3k}$ and $(a,q)=1$ such that $$\left|\alpha-\frac{a}{q} \right|<\frac{Z^{3k}}{qM^k}. $$ Letting $\beta=\alpha-a/q$, we have by Lemma \ref{Sasym}, as well as Lemma \ref{gradconst}, Proposition \ref{idzero}, and  Lemma \ref{content}, that \begin{equation} \label{Sproofmaj} S(\alpha)=\frac{w}{w_qq^{\l}} \sum_{\boldsymbol{s}\in \{0,\dots,q-1\}^{\ell} \cap W^{q}(Y)}e^{2\pi i g(\boldsymbol{s})a/q}\int_{[0,M]^{\ell}}e^{2\pi i g(\bsx)\beta}d\bsx + O_h\left(qM^{\ell-1}\log^C(Y) Z^{4k\l}\right)  ,\end{equation} where $$w_q=\prod_{\substack{ p\leq Y \\ p^{\gamma(p)}\mid q}}\left(1-\frac{j_d(p)}{p^{\gamma_d(p)\l}}\right)\gg_h 1. $$ Combining (\ref{Sproofmaj}) with Corollary \ref{qcor}, Lemma \ref{content}, and Corollary \ref{vdccor} yields (\ref{SmajII}) if $$q\leq Q\text{ and }|\beta|<\gamma,$$ as well as (\ref{SminII}) if $$q\leq Q \text{ and }|\beta|\geq \gamma \quad \text{or} \quad Q\leq q\leq Z^{3k}.$$ For this latter conclusion, when applying Corollary \ref{qcor} we use standard estimates that assure $$C^{\omega(q)}\tau(q)^{\l}\ll_{k,\l,\epsilon} q^{\epsilon} $$ for all $\epsilon>0$. Finally, it follows from Lemma \ref{weyl3} and Proposition \ref{content} that (\ref{SminII}) holds whenever $Z^{3k}\leq q \leq M^k/Z^{3k}$.  \qed 

\section*{Acknowledgments} 
The authors would like to thank Neil Lyall, \'Akos Magyar, Steve Gonek, and Paul Pollack for their helpful conversations and references. The authors would also like to thank the anonymous referee for their comments and suggestions. The second author would like to thank Gouquan Li for alerting him to an oversight in the proof of Lemma 4.5 in \cite{ricemax}, which is rectified in the proof of Lemma \ref{weyl3} in this paper.


\setlength{\parskip}{0pt}

\bibliographystyle{amsplain}

\begin{thebibliography}{10}    
\bibitem{BPPS} 
{\sc A. Balog, J. Pelik\'an, J. Pintz, E. Szemer\'edi}, {\em Difference sets without $\kappa$-th powers}, Acta. Math. Hungar. 65 (2) (1994), 165-187.

\bibitem{birch}
{\sc B. Birch}, {\em Forms in many variables}, Proc. Royal Soc. London. Ser. A. 265.1321 (1962), 245-263.

\bibitem{BloomMaynard}
{\sc T. Bloom, J. Maynard}, {\em A new upper bound for sets with no square differences}, preprint (2020), {\tt arxiv:2011.13266}.

\bibitem{Chen}
{\sc J.R. Chen}, {\em On Professor Hua's estimate of exponential sums}, Sci. Sinica 20 (1977), 711-719.

\bibitem{Conrad}
{\sc K. Conrad}, {\em A multivariable Hensel's lemma},\\ https://kconrad.math.uconn.edu/blurbs/gradnumthy/multivarhensel.pdf

\bibitem{cookmagyar}
{\sc B. Cook, A. Magyar}, {\em Diophantine equations in the primes}, Inven. Math. 198 (2014), 701-737.

\bibitem{CLR}
{\sc E. Croot, N. Lyall, A. Rice}, {\em Polynomials and primes in generalized arithmetic progressions}, Int. Math. Res. Not., no. 15 (2015), 6021-6043.

\bibitem{Deligne}
{\sc P. Deligne}, {\em La conjecture de Weil I}, Pub. Math. I.H.E.S. 43 (1974), 273-307.


\bibitem{Fulton} {\sc W. Fulton}, {\em Intersection theory}, Second edition, Ergebnisse der Mathematik und ihrer Grenzgebiete. 3. Folge. A Series of Modern Surveys in Mathematics {\bf 2}, Springer-Verlag, Berlin, 1998.

\bibitem{Furst}
{\sc H. Furstenberg}, {\em Ergodic behavior of diagonal measures and a theorem of {S}zemer\'edi on arithmetic progressions},
  J. d'Analyse Math, 71 (1977), 204-256.
  
\bibitem{Green} 
{\sc B. Green}, {\em On arithmetic structures in dense sets of integers}, Duke Math. Jour. 114 (2002) no.2, 215-238.

\bibitem{taoblog}
{\sc B. Green, T. Tao, T. Ziegler}, {\em A Fourier-free proof of the Furstenberg-S\'ark\"ozy theorem}, https://terrytao.wordpress.com/2013/02/28/a-fourier-free-proof-of-the-furstenberg-sarkozy-theorem/.

\bibitem{HallTen}
{\sc R. Hall, G. Tenanbaum}, {\em Divisors}, Cambridge Tracts in Mathematics, vol. 90, Cambridge University Press, 1990.
 
\bibitem{HLR}  
{\sc M. Hamel, N. Lyall, A. Rice}, {\em Improved bounds on S\'ark\"ozy's theorem for quadratic polynomials}, Int. Math. Res. Not. no. 8 (2013), 1761-1782

\bibitem{Hartshorne}
{\sc R. Hartshorne}, {\em Algebraic geometry}, Grad. Texts in Math. {\bf 52}, Springer-Verlag, New York-Heidelberg, 1977.

\bibitem{HindrySilverman} {\sc M. Hindry and J. H. Silverman}, {\em Diophantine geometry: an introduction}, Grad. Texts in Math. {\bf 201}, Springer-Verlag, New York, 2000.  

\bibitem{Jouanolou} {\sc J.-P. Jouanolou}, {\em Th\'{e}or\`emes de {B}ertini et applications}, Progress in Mathematics {\bf 42}, Birkh\"auser, Boston, 1983.

\bibitem{KMF}
{\sc T. Kamae, M. Mend\`es France}, {\em van der Corput's difference theorem}, Israel J. Math. 31, no. 3-4, (1978), pp. 335-342.

\bibitem{LangWeil} {\sc S. Lang and A. Weil}, {\em Number of points of varieties in finite fields}, Amer. J. Math. 76 (1954),  819-827.

\bibitem{Lewko}
{\sc M. Lewko}, {\em An improved lower bound related to the S\'ark\"ozy-Furstenberg Theorem}, Electron. J. Combin. 22 (2015), No. 32, 1-6.

\bibitem{lipan} 
{\sc H.-Z. Li, H. Pan}, {\em Difference sets and polynomials of prime variables}, Acta. Arith. 138, no. 1 (2009), 25-52.

\bibitem{LM} 
{\sc N. Lyall, \`A. Magyar}, {\em Polynomial configurations in difference sets}, J. Number Theory 129 (2009), 439-450.
 
\bibitem{Lucier}
{\sc J. Lucier}, {\em Intersective sets given by a polynomial}, Acta Arith. 123 (2006), 57-95.

\bibitem{Lucier2}
{\sc J. Lucier}, {\em Difference sets and shifted primes}, Acta. Math. Hungar. 120 (2008), 79-102.

\bibitem{Lyall}
{\sc N. Lyall}, {\em A new proof of S\'ark\"ozy's theorem}, Proc. Amer. Math. Soc. 141 (2013), 2253-2264.

\bibitem{PSS}
{\sc J. Pintz, W. L. Steiger, E. Szemer\'edi}, {\em On sets of natural numbers whose difference set contains no squares}, J. London Math. Soc. 37 (1988),  219-231.

\bibitem{PoonenSlavov}
{\sc B. Poonen, K. Slavov}, {\em The exceptional locus in the Bertini irreducibility theorem for a morphism}, preprint (2020), {\tt arXiv:2001.08672v2}.

\bibitem{ricemax} {\sc A. Rice}, {\em A maximal extension of the best-known bounds for the Furstenberg-S\'ark\"ozy Theorem}, Acta Arith. 187 (2019), 1-41.

\bibitem{thesis} 
{\sc A. Rice}, {\em Improvements and extensions of two theorems of S\'ark\"ozy}, Ph.D. thesis, University of Georgia, 2012. http://alexricemath.com/wp-content/uploads/2013/06/AlexThesis.pdf. 

\bibitem{Rice} 
{\sc A. Rice}, {\em S\'ark\"ozy's theorem for $\P$-intersective polynomials}, Acta Arith. 157 (2013), no. 1, 69-89.

\bibitem{Ricebin}
{\sc A. Rice}, {\em Binary quadratic forms in difference sets}, Combinatorial and Additive Number Theory III, Springer Proc. of Math. and Stat., vol. 297 (2020), 175-196.

\bibitem{Ruz}
{\sc I. Ruzsa, T. Sanders}, {\em Difference sets and the primes}, Acta. Arith. 131, no. 3 (2008), 281-301.

\bibitem{Ruz2}
{\sc I. Ruzsa}, {\em Difference sets without squares}, Period. Math. Hungar. 15 (1984), 205-209.

\bibitem{Ruz3}
{\sc I. Ruzsa}, {\em On measures on intersectivity}, Acta Math. Hungar. 43(3-4) (1984), 335-340. 
 
\bibitem{Sark1}
{\sc A. S\'ark\"ozy}, {\em On difference sets of sequences of integers I}, Acta. Math. Hungar. 31(1-2) (1978), 125-149.

\bibitem{Sark3}
{\sc A. S\'ark\"ozy}, {\em On difference sets of sequences of integers III}, Acta. Math. Hungar. 31(3-4) (1978), 355-386.


\bibitem{Slip} {\sc S. Slijep\v{c}evi\'c}, {\em A polynomial S\'ark\"ozy-Furstenberg theorem with upper bounds}, Acta Math. Hungar. 98 (2003),  275-280.

\bibitem{vaughan}
{\sc R. C. Vaughan}, {\em The Hardy-Littlewood method}, Cambridge University Press, Second Edition, 1997.

\bibitem{wang}
{\sc R. Wang}, {\em On a theorem of S\'ark\"ozy for difference sets and shifted primes}, Journal of Number Theory, Volume 211 (2020), 220-234.

\bibitem{wessel}
{\sc M. Wessel}, {\em An algebraic interpretation of the polynomial Szemer\'edi theorem}, Universiteit Utrecht Bachelor thesis (2020).

\bibitem{younis}
{\sc K. Younis}, {\em Lower bounds in the polynomial Szemer\'edi theorem}, preprint (2019), {\tt arXiv:1908.06058}.

\

\end{thebibliography}

%
%

\begin{dajauthors}
\begin{authorinfo}[john]
  John R. Doyle\\
  Department of Mathematics\\
  Oklahoma State University\\
  Stillwater, OK 74078, USA\\
  john.r.doyle\imageat{}okstate\imagedot{}edu \\
  \url{https://math.okstate.edu/people/jdoyle/}
\end{authorinfo}
\begin{authorinfo}[alex]
  Alex Rice\\
  Department of Mathematics\\
  Millsaps College\\
  Jackson, MS 39210, USA\\
  riceaj\imageat{}millsaps\imagedot{}edu\\
  \url{https://alexricemath.com/}
\end{authorinfo}
\end{dajauthors}

\end{document}